\documentclass[11pt,notitlepage]{report}
\usepackage[english]{babel}
\usepackage{url}
\usepackage[utf8x]{inputenc}
\usepackage{amsmath}
\usepackage{amssymb}
\usepackage{graphicx}

\graphicspath{{images/}}
\usepackage{parskip}
\usepackage{fancyhdr}
\usepackage{vmargin}
\usepackage{mathabx}
\usepackage{subcaption}
\usepackage{fancyvrb}
\usepackage{bm}
\usepackage{amsthm}
\usepackage{csquotes}
\usepackage{bbm}
\usepackage[round,authoryear,sort]{natbib}
\usepackage{xcolor}
\usepackage{enumerate}

\setmarginsrb{1.75 cm}{1 cm}{1.75 cm}{2.5 cm}{1 cm}{1.5 cm}{1 cm}{1.5 cm}

\title{Limit theorems for the realised semicovariances of multivariate Brownian semistationary processes}
\author{Yuan Li\thanks{E-mail: yuan.li18@imperial.ac.uk},\  Mikko S. Pakkanen\thanks{E-mail: m.pakkanen@imperial.ac.uk},\  Almut E.D. Veraart\thanks{E-mail: a.veraart@imperial.ac.uk}}
\date{}

\makeatletter
\let\thetitle\@title
\let\theauthor\@author
\let\thedate\@date
\makeatother

\newtheorem{prop}{Proposition}
\newtheorem{thm}{Theorem}
\newtheorem{Def}{Definition}
\newtheorem{lem}{Lemma}

\newtheorem{rmk}{Remark}
\newtheorem{asm}{Assumption}

\numberwithin{thm}{subsection}
\numberwithin{prop}{subsection}
\numberwithin{Def}{subsection}
\numberwithin{cla}{subsection}
\numberwithin{rmk}{subsection}
\numberwithin{asm}{subsection}
\numberwithin{cor}{subsection}
\numberwithin{exa}{subsection}
\numberwithin{lem}{subsection}
\usepackage{titling}

\usepackage{setspace}
\usepackage{atbegshi}
\AtBeginDocument{\AtBeginShipoutNext{\AtBeginShipoutDiscard}}
\AtBeginDocument{\addtocounter{page}{-1}}
\begin{document}

\begin{center}
\vspace{-1cm}
\maketitle
\vspace{-2cm}
\emph{Department of Mathematics, Imperial College London, SW7 2AZ, UK}
\rule{18cm}{1pt}
\end{center}
\textbf{Abstract}\\
In this article, we will introduce the realised semicovariance for Brownian semistationary (BSS) processes, which is obtained from the decomposition of the realised covariance matrix into components based on the signs of the returns and study its in-fill asymptotic properties. More precisely, weak convergence in the space of c\`{a}dl\`{a}g functions endowed with the Skorohod topology for the realised semicovariance of a general Gaussian process with stationary increments is proved first. The proof is based on the Breuer-Major theorem and on a moment bound for sums of products of non-linearly transformed Gaussian vectors. Furthermore, we establish a corresponding stable convergence. Finally, a central limit theorem for the realised semicovariance of multivariate BSS processes is established. These results extend the limit theorems for the realised covariation to a result for non-linear functionals.\\

\rule{18cm}{1pt}

\pagebreak

\renewcommand{\thesection}{\arabic{section}}
\section{Introduction}
The idea of studying the realised semicovariance for Brownian semistationary processes is inspired by the article \cite{TB2020}, where the authors originally proposed the realised semicovariance, which is obtained from the decomposition of the realised covariance matrix into components based on the signs of the increments, and studied its in-fill asymptotic (which means that the time step between observations goes to zero) properties for semimartingales. Intuitively, a statistic composed of positive increments and one composed of negative increments should carry distinct economic information, which has been studied and illustrated empirically in \cite{TB2020}. Furthermore, they proved that using semicovariances significantly improved volatility forecasts for certain stock markets. Thus, we want to extend their work to an alternative, non-semimartingale setting. Such more general settings are of interest in financial applications when transaction costs and market microstructure noise are to be taken into account, see \cite{MP2011}. Hence we devote our attention to the Brownian semistationary (BSS) process, which was first introduced in \cite{BN20091} and is not necessarily a semimartingale, and is interesting because of its various applications, especially in turbulence \citep{CH2013}, finance \citep{MP2011} and energy \citep{BN2013}. It is a novel topic to study the in-fill asymptotic properties of realised semivariance/ semicovariance for a BSS process since it includes the cases outside the semimartingale class. To present this idea, we let a bivariate BSS process have the form
$$Y_t^{(i)}=\int_{-\infty}^tg^{(i)}(t-s)\sigma^{(i)}_sdW^{(i)}_s,\ i=1,2.$$
Suppose that we observe this process at a fixed frequency ($\frac{1}{n}$) over a time interval $[0,T]$; then its realised semicovariance can be written as 
$$\sum_{i=1}^{[nt]}p(\Delta_i^n Y^{(1)})p(\Delta_i^n Y^{(2)}),\ \forall t\in[0,T],$$
where $\Delta_i^nY^{(j)}:=Y^{(j)}_{i/n}-Y^{(j)}_{(i-1)/n},\ j=1,2,$ and $p(x)=\max\{x,0\}$. It simplifies to semivariance in the univariate case for identical superscripts. In the realm of in-fill asymptotic properties of realised semivariance/semicovariance, \cite{BN2010} first studied the realised semivariance based on one univariate semimartingale model and have shown that the realised semivariance has important predictive qualities for future market volatility. \cite{JJ2012} have established the in-fill asymptotic theory for a family of functions of increments for a class of It\^o semimartingale. Based on the previous works, \cite{TB2020} extend the semivariances to a multivariate setting, i.e., realised semicovariance. The difficulty of our work comes from two parts: the first one is that the methods in the previously mentioned references are based on semimartingale techniques which are impossible to apply to general BSS models; the second one is that when we extend the univariate case to the multivariate case, the complexity of the question is greatly increased. We will discuss these difficulties in detail in the following sections.\\

There is a series of papers \citep{BN2009,BN2011} that have studied the realised power variation for the univariate BSS process. However, they do not cover the present question because the function of the increments we will work with for the univariate case is $f(x)=x^21_{\{x\ge0\}}$, which is not an even function and makes a significant difference in the underlying theory. There are only very few existing results for the case of multivariate BSS processes. \cite{AA2019} and \cite{RA2019} can be viewed as the starting point of using multivariate BSS processes in stochastic modelling, where they established the asymptotic properties of the realised covariance for multivariate BSS processes. Our work will extend the limit theorems for the realised covariation to a version for the non-linear functionals.\\

This article is structured as follows. In Section 2.1, we first introduce the settings for univariate BSS processes and some general assumptions. The selected tools from Malliavin calculus are reviewed in Section 2.2. We formulate our general assumptions in Section 2.3. The key result, the central limit theorem for the realised semivariance of a univariate BSS process, can be found in Section 2.4. Similar to the univariate case, we introduce the settings for the multivariate case in Section 3.1 and state the key results in Section 3.2. We will discuss our current results and potential future works in Section 4. For ease of exposition, all proofs will be postponed to Section 5.

\section{Univariate case: Realised semivariance of a BSS process}
We start with the univariate case by introducing the basic settings, mathematical tools which will be used, technical assumptions, and general procedure in proving in-fill asymptotic theorems in this section.
\subsection{Setting}
Let $(\Omega,\mathcal{F},\mathcal{F}_t,\mathbb{P})$ denote a complete, filtered probability space. We denote by $[0,T]$ a finite time interval for some $T>0$ and by $\mathcal{B}(\mathbb{R})$ the class of Borel subsets of $\mathbb{R}$. A stochastic process $(X_t)_{t\in\mathbb{R}}$ is $\mathcal{F}_t$-adapted if $X_t\in\mathcal{F}_t$. We first introduce the Brownian measure.\\
\begin{Def}[Brownian measure]
An $\mathcal{F}_t$-adapted Brownian measure $W:\Omega\times\mathcal{B}(\mathbb{R})\rightarrow\mathbb{R}$ is a Gaussian stochastic measure such that, if $A\in\mathcal{B}(\mathbb{R})$ with $Leb(A)<\infty$, then $W(A)\sim N(0,Leb(A))$, where $Leb(\cdot)$ is the Lebesgue measure. Moreover, if $A\subseteq[t,\infty)$, then $W(A)$ is independent of $\mathcal{F}_t$.
\end{Def}
We will use $W$ to denote a Brownian measure in this article. The Gaussian core which is defined below is a Gaussian moving average process and is useful for our purposes.
\begin{Def}[The Gaussian core]\label{gaussiancore} A Gaussian process with stationary increments $G$ is defined as
\begin{equation}
G_t=\int_{-\infty}^tg(t-s)dW_s,
\end{equation}
where $g$ is a square-integrable deterministic function on $\mathbb{R}$ with $g(t)=0$ for $t\le0$, and $W$ is a Brownian measure adapted to $\mathcal{F}_t$. $G$ is called the Gaussian core.\end{Def}
By introducing stochastic volatility to the Gaussian core, we will have a Brownian semistationary process defined below.\\
\begin{Def}[Brownian semistationary process]
Let $\sigma$ be a $\mathcal{F}_t$-adapted c\`adl\`ag process, and assume that the function $g$ is continuously differentiable on $(0,\infty)$, $|g'|$ is non-increasing on $(b,\infty)$ for some $b>0$ and $g'\in L^2((\epsilon,\infty))$ for any $\epsilon>0$. Then we define our Brownian semistationary process
\begin{equation}
X_t=\int_{-\infty}^tg(t-s)\sigma_sdW_s.
\end{equation}
We also require $\int_{-\infty}^tg^2(t-s)\sigma_s^2ds<\infty\ a.s.$ to ensure that $X_t<\infty\ a.s.$ for all $t\ge0$. Moreover, we assume that for any $t>0$,
\begin{equation}\label{smoothness}
F_t=\int_{1}^\infty(g'(s))^2\sigma_{t-s}^2ds<\infty,\ a.s.
\end{equation}
\end{Def}
\begin{rmk}
The condition (\ref{smoothness}) is to control the moments of increments of the BSS process by the moments of increments of the corresponding Gaussian core, see Lemma 1 in \cite{BN2011}.
\end{rmk}
Let $f(x)=x^21_{\{x\ge0\}}$. The normalised upside realised semivariance is defined as
$$V(X,f)_t^n=\frac{1}{n}\sum_{i=1}^{[nt]}f\left(\frac{\Delta_i^nX}{\tau_n}\right)=\frac{1}{n\tau_n^2}\sum_{i=1}^{[nt]}(\Delta_i^nX)^21_{\{\Delta_i^nX\ge0\}},$$
where $\Delta_i^nX=X_{\frac{i}{n}}-X_{\frac{(i-1)}{n}}$, and $\tau_n^2={R}(\frac{1}{n})$ with ${R}(t)=\mathbb{E}[|G_{t+s}-G_s|^2],\ t\ge0$.\\
Similarly, we define 
$$V(G,f)_t^n=\frac{1}{n}\sum_{i=1}^{[nt]}f\left(\frac{\Delta_i^nG}{\tau_n}\right)=\frac{1}{n\tau_n^2}\sum_{i=1}^{[nt]}(\Delta_i^nG)^21_{\{\Delta_i^nG\ge0\}},$$
where $\Delta_i^nG=G_{\frac{i}{n}}-G_{\frac{(i-1)}{n}}.$
\subsection{Wiener-It\^{o} chaos decomposition}
Similar to most articles which studied the in-fill asymptotic properties of BSS processes, our work heavily relies on the Wiener-It\^{o} chaos decomposition, see the monograph by \cite{NP2012}. Here we briefly introduce some basic concepts and results which will be used in this paper.\\
\begin{Def}[Isonormal Gaussian processes] Fix a real separable Hilbert space $\mathcal{H}$, with inner product $\langle\cdot,\cdot\rangle_{\mathcal{H}}$ and norm $\langle\cdot,\cdot\rangle_{\mathcal{H}}^{1/2}=\|\cdot\|_{\mathcal{H}}$. We write $Z=\{Z(h):h\in\mathcal{H}\}$ to indicate an isonormal Gaussian process over $\mathcal{H}$. This means that $Z$ is a centered Gaussian family, defined on probability space $(\Omega,\mathcal{F},P)$ and such that $\mathbb{E}[Z(g)Z(h)]=\langle g,h\rangle_{\mathcal{H}}$ for every $g,h\in\mathcal{H}$.
\end{Def}
\begin{prop}[Proposition 2.1.1 in \cite{NP2012}]\label{isonormalexist}
Given a real separable Hilbert space $\mathcal{H}$, there exists an isonormal Gaussian process over $\mathcal{H}$.
\end{prop}
According to our settings, we let $\mathcal{H}$ be the Hilbert space generated by the rescaled increments of the Gaussian core:
$$\left(\frac{\Delta_i^nG}{\tau_n}\right)_{n\ge1,1\le i\le [nT]}$$
equipped with the inner product $\langle X,Y\rangle_{\mathcal{H}}=\mathbb{E}[XY]$ for $X,Y\in\mathcal{H}$. By Proposition \ref{isonormalexist}, we have an isonormal process $Z$ over this Hilbert space $\mathcal{H}$. For the rest of this paper, we will assume that $\mathcal{F}$ is the $\sigma$-algebra generated by $Z$. Next, we introduce Hermite polynomials which are an orthogonal basis of $L^2(\mathbb{R},\gamma)$, the space of square-integrable functions on $\mathbb{R}$ with respect to the standard Gaussian measure $\gamma$.

\begin{Def}[Hermite polynomials] Let $p\ge 0$ be an integer. We define the $p$th Hermite polynomial as $H_0=1$ and $H_{p+1}(x)=xH_p(x)-pH_{p-1}(x)$, where we use the convention that $H_{-1}(x)=0$.

\end{Def} 
The Wiener chaos plays a crucial role in $L^2(\Omega,\mathcal{F},P)$, which is analogous to that of the Hermite polynomials for $L^2(\gamma)$, since any element of $L^2(\Omega,\mathcal{F},P)$ has a unique decomposition in terms of the Wiener chaos expansion, which will be explained by the following Wiener-It\^{o} chaos decomposition theorem.
\begin{Def}[Wiener chaos] For each $n\ge0$, we write $\mathcal{H}_n$ to denote the closed linear subspace of $L^2(\Omega,\mathcal{F},P)$ generated by the random variables of type $H_n(X(h)),\ h\in\mathcal{H},\ \|h\|_{\mathcal{H}}=1$. The space $\mathcal{H}_n$ is called the $n$th Wiener chaos of $X$.

\end{Def}

\begin{thm}[Wiener-It\^{o} chaos decomposition, Theorem 2.2.4 in \cite{NP2012}] One has that $L^2(\Omega,\mathcal{F},P)=\bigoplus_{n=1}^\infty \mathcal{H}_n$. This means that every random variable $F\in L^2(\Omega,\mathcal{F},P)$ admits a unique expansion of the type $F=\mathbb{E}[F]+\sum_{n=1}^\infty F_n$, where $F_n\in\mathcal{H}_n$ and the series converges in $L^2(\Omega,\mathcal{F},P)$.

\end{thm}

Given an integer $q\ge2$, we denote by $\mathcal{H}^{\otimes q}$ and $\mathcal{H}^{\odot q}$, respectively, the $q$th tensor product and the $q$th symmetric tensor product of $\mathcal{H}$.

\begin{Def}[Contractions] Let $g=g_1\otimes\cdots\otimes g_n\in\mathcal{H}^{\otimes n}$ and $h=h_1\otimes\cdots\otimes h_m\in\mathcal{H}^{\otimes m}$. For any $0\le p\le n\land m$, we define the $p$th contraction $g\otimes_p h$ as the element of $\mathcal{H}^{\otimes n+m-2p}:g\otimes_p h:=\langle g_1,h_1\rangle_{\mathcal{H}}\cdots \langle g_p,h_p\rangle_{\mathcal{H}}g_{p+1}\otimes\cdots\otimes g_n\otimes h_{p+1}\otimes\cdots\otimes h_m$. We denote by $g\tilde\otimes_ph$ its symmetrisation. When $n=m=p$, we denote $\langle g,h\rangle_{\mathcal{H}^{\otimes p}}:=g\otimes_p h=\langle g_1,h_1\rangle_{\mathcal{H}}\cdots \langle g_p,h_p\rangle_{\mathcal{H}}$ and $\|\cdot\|_{\mathcal{H}^{\otimes p}}:=\langle\cdot,\cdot\rangle_{\mathcal{H}^{\otimes p}}^{1/2}$.

\end{Def}

Multiple integrals in the Malliavin calculus setting will help us to establish a connection between symmetric tensor products and Wiener chaos. For an integer $p\ge1$, we denote the $p$th multiple integral by $I_p:\ \mathcal{H}^{\odot p}\rightarrow \mathcal{H}_p$ the isometry from the symmetric tensor product $\mathcal{H}^{\odot p}$ to the $p$th Wiener chaos $\mathcal{H}_p$, equipped with the norm $\sqrt{p!}\|\cdot\|_{\mathcal{H}^{\otimes p}}$. 
\begin{thm}[Theorem 2.7.7 in \cite{NP2012}] Let $f\in\mathcal{H}$ be such that $\|f\|_{\mathcal{H}}=1$. Then, for any integer $p\ge1$, we have 
\begin{equation}
H_p(Z(f))=I_p(f^{\otimes p}).
\end{equation}

\end{thm}
Multiple integrals of different orders are orthogonal.
\begin{prop}[Isometry property of integrals, Proposition 2.7.5 in \cite{NP2012}]
Fix integers $1\le q\le p$, as well as $f\in\mathcal{H}^{\odot p}$ and $g\in\mathcal{H}^{\odot q}$. We have
\begin{equation}
\mathbb{E}[I_p(f)I_q(g)]=\begin{cases}
p!\langle f,g\rangle_{\mathcal{H}^{\otimes p}}, & \text{if }p=q,\\
0, & \text{otherwise.}
\end{cases}
\end{equation}
\end{prop}
The following theorem is the so-called fourth moment theorem which provides equivalent conditions for the convergence of a multiple integral to the standard Gaussian distribution.
\begin{thm}[Fourth-moment theorem, Theorem 5.2.7 in \cite{NP2012}] Let $k\ge2$, $f_n\in\mathcal{H}^{\odot k}$ for any $n\in\mathbb{N}$, and suppose that
$$\mathbb{E}[I_k(f_n)^2]=k!\|f_n\|^2_{\mathcal{H}^{\otimes k}}\overset{n\rightarrow\infty}\longrightarrow1.$$
Then the following conditions are equivalent:\\
(a) $I_k(f_n)\overset{\mathcal{L}}\longrightarrow N(0,1)$, $n\rightarrow\infty$,\\
(b) $\mathbb{E}[I_k(f_n)^4]\longrightarrow 3$, $n\rightarrow\infty$,\\
(c) $\|f_n\otimes_r f_n\|^2_{\mathcal{H}^{\otimes 2k-2r}}\longrightarrow0$ for $1\le r \le k-1$, $n\rightarrow\infty$.
\end{thm}
Based on the condition (c) above and the orthogonality of Wiener chaos, we will have a generalised CLT for a sequence of random variables that admit Wiener chaos decomposition.
\begin{thm}\label{GFMT} Suppose that for any $n\in\mathbb{N}$, we have $f_{k,n}\in\mathcal{H}^{\odot k},k\in\mathbb{N}$. If
\begin{enumerate}[(i)] 
\item  $\lim_{m\rightarrow\infty}\limsup_{n\rightarrow\infty}\sum_{k=m}^\infty k!\|f_{k,n}\|_{\mathcal{H}^{\otimes k}}^2=0$,
\item $k!\|f_{k,n}\|_{\mathcal{H}^{\otimes k}}^2\overset{n\rightarrow\infty}{\longrightarrow}\sigma_k^2$ for any $k\in\mathbb{N}$, so that $\sigma^2:=\sum_{k=1}^\infty\sigma_k^2<\infty$,
\item $\|f_{k,n}\otimes_r f_{k,n}\|_{\mathcal{H}^{\otimes 2k-2r}}^2\overset{n\rightarrow\infty}{\longrightarrow}0$ for any $1\le r \le k-1$ and $k\ge2$,\\
\end{enumerate}
then 
$$\sum_{k=1}^\infty I_k(f_{k,n})\overset{\mathcal{L}}{\longrightarrow}N(0,\sigma^2),\ n\rightarrow\infty.$$
\end{thm}
And for multivariate case, we have the following theorem.
\begin{thm}[Multivariate central limit theorem, Theorem 5 in \cite{BN2009}]\label{MGFMT} Let $d\ge2$ and $F_k$ be a $d$-dimensional random vector $F_n=(Y_n^{(1)},...,Y_n^{(d)})^T$ for any $n\in\mathbb{N}$, where the superscript $T$ denotes the transpose of a vector. Assume that $F_n$ has a chaos representation
$$F_n^{(i)}=\sum_{m=1}^\infty I_m(f_{m,n}^{(i)}),\ \ \ i=1,...,d,$$
with $f_{m,n}^{(i)}\in\mathcal{H}^{\odot m}$. Suppose that the following conditions hold:
\begin{enumerate}[(i)]
\item For any $i=1,...,d$ we have $\lim_{N\rightarrow\infty}\limsup_{n\rightarrow\infty}\sum_{m=N}^\infty m!\|f_{m,n}^{(i)}\|_{\mathcal{H}^{\otimes m}}^2=0$.
\item For any $m\ge1$, $i,j=1,...,d$ we have constants $\Sigma_{ij}^m$ such that 
$$\lim_{n\rightarrow\infty}\mathbb{E}\left[I_m(f_{m,n}^{(i)})I_m(f_{m,n}^{(j)})\right]=\lim_{n\rightarrow\infty}\langle f_{m,n}^{(i)},f_{m,n}^{(j)}\rangle_{\mathcal{H}^{\odot m}}=\Sigma_{ij}^m,$$
and the matrix $\Sigma^m=(\Sigma_{ij}^m)_{1\le i,j\le d}$ is positive definite for all $m$.
\item $\sum_{m=1}^\infty \Sigma^m=\Sigma\in\mathbb{R}^{d\times d}$.
\item For any $m\ge1,i=1,...,d$ and $p=1,...,m-1$
$$\lim_{n\rightarrow\infty}\|f_{m,n}^{(i)}\otimes_p f_{m,n}^{(i)}\|_{\mathcal{H}^{\otimes 2(m-p)}}^2=0.$$
\end{enumerate}
Then we have $$F_n\overset{\mathcal{L}}\longrightarrow N_d(0,\Sigma),\ n\rightarrow\infty.$$
\end{thm}
\subsection{Technical assumptions}
The aim of this paper is to derive an asymptotic theory for the realised semivariance. To this end, we need to make some technical assumptions. Generally, the assumptions for the central limit theorem (CLT) should be stronger than those for the weak law of large numbers (WLLN) and the proof of the CLT will be more difficult. For the sake of brevity, we will only focus on the CLT in this article and will explain which assumptions can be relaxed in order to establish the WLLN. Here we begin with some general assumptions.\\
\begin{asm}\label{slowlyvarying}
There exist slowly varying (at $0$) functions $L_0(t)$ and $L_2(t)$ which are continuous on $(0,\infty)$ such that
\begin{equation}
{R}(t)=t^{2\alpha+1}L_0(t),
\end{equation}
and
\begin{equation}
{R}''(t)=t^{2\alpha-1}L_2(t),
\end{equation}
where $\alpha\in(-\frac{1}{2},\frac{1}{2})\setminus\{0\}$. Furthermore, there exist $b\in(0,1)$ such that
\begin{equation}
\limsup_{x\rightarrow0^+}\sup_{y\in[x,x^b]}\left|\frac{L_2(y)}{L_0(x)}\right|<\infty.
\end{equation}
\end{asm}
The above assumption helps us control the correlations $r_n(j):=\mathbb{E}\left[\frac{\Delta_1^nG}{\tau_n}\frac{\Delta_{1+j}^nG}{\tau_n}\right]$ and provides us with a dominant sequence as shown in the following lemma.\\
\begin{lem}[Theorem 4.1 in \cite{AA2019}]
For any $\epsilon>0$ with $\epsilon<1-2\alpha$, we can define
$$r(j)=(j-1)^{2\alpha+\epsilon-1},j\ge2,\ r(0)=r(1)=1.$$
Under the above assumption, there exists a natural number $n_0(\epsilon)$ such that 
\begin{equation}
|r_n(j)|\le Cr(j),\ j\ge0,
\end{equation}
for all $n\ge n_0(\epsilon)$. Then we can have a dominant sequence of $r_n(j)$ by letting $\bar{r}(j)=Cr(j)$.\\
Moreover, define $\rho_\alpha(j)=\frac{1}{2}\left((j-1)^{2\alpha+1}-2j^{2\alpha+1}+(j+1)^{2\alpha+1}\right)$ for $j\ge1$ and $\rho_\alpha(0)=1$, it holds that
\begin{equation}
r_n(j)\rightarrow\rho_\alpha(j),\ n\rightarrow\infty.
\end{equation}
\end{lem}
The next assumption will help us to have a unique limit when our functional converges. See \citet[Remark 6]{BN2011} for details.
\begin{asm}\label{kernelmeasure} Suppose that
$$\lim_{n\rightarrow\infty}\pi^n((\epsilon,\infty))=0,\ \forall\epsilon>0,$$
where $$\pi^n(A)=\frac{\int_A(g(x-1/n)-g(x))^2dx}{\int_0^\infty(g(x-1/n)-g(x))^2dx},\ A\in\mathcal{B}(\mathbb{R}).$$
\end{asm}

\subsection{The central limit theorem for the realised semivariance}
In order to set up a central limit theorem for the realised semivariance of BSS processes, it is crucial to prove the CLT for the realised semivariance of the Gaussian core first. Hence, we will adopt the following assumption.\\
\begin{asm}\label{lesszero}
Given Assumption \ref{slowlyvarying}, the parameter $\alpha$ satisfies $\alpha\in(-\frac{1}{2},0)$.
\end{asm}
\begin{rmk}
The above assumption guarantees the summability of the dominant sequence $\bar{r}(j)$, i.e.,
\begin{equation}
\sum_{j=1}^\infty r_n(j)\le \sum_{j=1}^\infty\bar{r}(j)<\infty,\ \forall n\ge1.
\end{equation}
\end{rmk}
We will use $\mathcal{D}([0,T])$ to denote the space which is the set of all c\`adl\`ag functions from $[0,T]$ to $\mathbb{R}$ and use the notation $‘\overset{\mathcal{L}}{\rightarrow}${’},$\ ‘\overset{st.}{\rightarrow}${’} for weak convergence and stable convergence, respectively. Recall that a sequence of stochastic process $\{X^{(n)}\}$ converges weakly to $X$, then $X^{(n)}\overset{st.}{\rightarrow}X$ if and only if for any bounded Borel function $f$ and any $\mathcal{F}$-measurable fixed variable $Z$, $\lim_{n\rightarrow\infty}\mathbb{E}[f(X^{(n)})Z]=\mathbb{E}[f(X)Z]$.\\
Given that the assumptions presented before hold, the CLT for the realised semivariance of the Gaussian core will be a special case of Theorem 1.1 in \cite{ID2018}.\\

\begin{thm}\label{weakGcore} If Assumption \ref{slowlyvarying}, \ref{kernelmeasure} and \ref{lesszero} hold, then we have
$$\sqrt{n}\left(V(G,f)_t^n-\frac{1}{2}t\right)\overset{\mathcal{L}}{\rightarrow}\sqrt\beta B_t,\ n\rightarrow\infty,$$
where $\beta:=\sum_{k=1}^\infty k!a_k^2(1+2\sum_{i=1}^\infty\rho_\alpha(i)^k)<\infty$, $a_k$ is the k-th coefficient of the Hermite expansion of function $f$, and $B$ is a Brownian motion. The weak convergence holds in $\mathcal{D}([0,T])$ equipped with uniform topology.
\end{thm}

\begin{rmk}
We can also prove the above theorem by the usual procedure: 1. Prove the convergence of finite-dimensional distributions by the Fourth Moment Theorem; 2. Prove the tightness by the approach of \cite{ID2018}.
\end{rmk}

So far, the viewpoint of our work is quite similar to the existing works of (multi-)power variation for BSS processes in \cite{BN2009} and \cite{BN2011}. In these articles, the CLT has been established for $V(G,H)$ where $G$ is the Gaussian core as we defined in Definition \ref{gaussiancore}, and function $H(x):=|x|^p-\mu_p$ where $\mu_p=\mathbb{E}[|Z|^p],\ Z\sim N(0,1)$. The main difference between our work and theirs comes from the Hermite rank for different functions. Notice that the function $\varphi(x):=f(x)-\frac{1}{2}=x^21_{\{x\ge0\}}-\frac{1}{2}$ has Hermite rank 1, whereas the Hermite rank of $H(x)$ is 2. Such a difference will result in changes of the dependence structure between the Gaussian core and the limiting Brownian motion $B$. Since both the first-order multiple integral of the Hermite expansion of $V(G,f)_t^n$ and the Gaussian core belong to the 1st Wiener chaos, $V(G,f)_t^n$ and the Gaussian core are no longer orthogonal. The following theorem establishes the asymptotic dependence structure between $V(G,f)_t^n$ and the Gaussian core.\\
\begin{thm}\label{multiclt} Let Assumption \ref{slowlyvarying}, \ref{kernelmeasure} and \ref{lesszero} hold, then we have
$$(G_t,\sqrt{n}(V(G,f)_t^n-\frac{1}{2}t))\overset{\mathcal{L}}{\rightarrow}(G_t,\sqrt\beta B_t),$$
where $B_t$ is a Brownian motion independent of the process $\{G_t-G_0\}$ and the weak convergence holds in $\mathcal{D}([0,T])^2$ equipped with the Skorohod topology.
\end{thm}
Note that in our settings, we already denote the $\sigma$-algebra generated by the process $\{G_t-G_0\}$ by $\mathcal{F}$, Theorem \ref{multiclt} is equivalent to the following stable convergence theorem by applying condition $D^{''}$ from Proposition 2 of \cite{AE1978}.\\
\begin{thm}\label{stclt} Under Assumption \ref{slowlyvarying}, \ref{kernelmeasure} and \ref{lesszero}, we have
$$\sqrt{n}(V(G,f)_t^n-\frac{1}{2}t)\overset{st.}{\rightarrow}\sqrt{\beta}B_t,\ n\rightarrow\infty,\ \text{w.r.t. }\sigma\text{-algebra }\mathcal{F}.$$
\end{thm}
After establishing the asymptotic theory for the Gaussian core, it remains to take the stochastic volatility $\sigma_t$ into consideration. We already assumed that the process $\sigma_t$ is c\`adl\`ag, which is not sufficient for the CLT for the realised semivariance of BSS processes. Here we make two additional assumptions on the process $\sigma$, which are used in the related literature. \\
\begin{asm}\label{Holdercont}
The volatility process $\sigma_t$ is non-negative, $\mathcal{F}$-measurable, and $\eta$-H\"older continuous.
\end{asm}
\begin{asm}\label{Holderadv}
We assume that $\eta>\frac{1}{2}$ and there exist a constant $\lambda<-1$ such that for any $\epsilon_n=O(n^{-\kappa}),\ \kappa\in(0,1)$, we have
$$\pi^n((\epsilon_n,\infty))=O(n^{\lambda(1-\kappa)}).$$
\end{asm}
Under the above assumptions, the final theorem, the CLT for the realised semivariance of BSS processes, can be proved by using the so-called blocking technique.\\
\begin{thm}[Central limit theorem]\label{clt} Let Assumption \ref{slowlyvarying}, \ref{kernelmeasure}, \ref{lesszero}, \ref{Holdercont} and \ref{Holderadv} hold. Then we have
$$\sqrt{n}\left(V(X,f)_t^n-\frac{1}{2}\int_0^t\sigma_s^2ds\right)\overset{\mathcal{L}}{\rightarrow}\sqrt\beta\int_{0}^t\sigma_s^2dB_s,\ n\rightarrow\infty,$$
in the Skorokhod space $\mathcal{D}([0,T])$ where $\beta=\sum_{k=1}^\infty k!a_k^2(1+2\sum_{i=1}^\infty\rho(i)^k)<\infty$, $B$ is a Brownian motion independent of the volatility process $\sigma$.
\end{thm}
\begin{rmk}
Notice that the weak law of large numbers (WLLN) for $V(X,f)_t^n$ will be a direct corollary of Theorem \ref{clt}. For a WLLN, we only need Assumption \ref{slowlyvarying} and Assumption \ref{kernelmeasure}, and all the other assumptions we have made in this section are not needed. 
\end{rmk}

\begin{rmk}
The above central limit theorem can be generalised to a class of functions $g$ consisting of the power function and indicator function $I(x)=1$ or $1_{\{x\ge0\}}$ or $1_{\{x\le0\}}$. Since the main focus of this paper is to develop CLT for semicovariances, and the univariate case can be viewed as a special case of the multivariate one, we will discuss the generalisation for the multivariate case in detail in the next section, see Theorem \ref{GBi_CLT}.
\end{rmk}

\section{Multivariate case: realised semicovariance of a bivariate BSS process}
\subsection{Setting}
The setup in the multivariate case will be a natural extension of the univariate case. We will only focus on the bivariate case since higher dimensions do not require any essential changes to the proof but result in a more complicated exposition.\\
\begin{Def}[The Gaussian cores]
Let $W^{(1)}$ and $W^{(2)}$ be two $\mathcal{F}_t$-adapted Brownian measures and jointly Gaussian. $g^{(1)}$, $g^{(2)}$ are two deterministic functions with the same properties in Definition \ref{gaussiancore}. The Gaussian cores are defined as
\begin{equation}
G_t^{(j)}=\int_{-\infty}^tg^{(j)}(t-s)dW_s^{(j)},\ j=1,2.
\end{equation}
Moreover, we assume that $W^{(1)}$ and $W^{(2)}$ satisfy $\mathbb{E}[dW_t^{(1)}dW_t^{(2)}]=\rho dt$, for $\rho\in[-1,1]$. Then it is possible to see that the bivariate Gaussian core $(G_t^{(1)},G_t^{(2)})$ is a stationary Gaussian process with stationary increments.
\end{Def}

\begin{Def}[Bivariate Brownian semistationary process] Let $\sigma^{(1)},\sigma^{(2)}$ be $\mathcal{F}_t$-adapted c\`adl\`ag processes, and assume that the function $g^{(j)}$ is continuously differentiable on $(0,\infty)$, $|g^{(j)}{'}|$ is non-increasing on $(b^{(j)},\infty)$ for some $b^{(j)}>0$ and $g^{(j)}{'}\in L^2((\epsilon,\infty))$ for any $\epsilon>0$, $j=1,2$. 
Then we define the Brownian semistationary processes as
\begin{equation}
Y_t^{(j)}=\int_{-\infty}^tg^{(j)}(t-s)\sigma_s^{(j)}dW_s^{(j)},\ j=1,2.
\end{equation}
We also require $\int_{-\infty}^tg^{(j)2}(t-s)\sigma_s^{(j)2}ds<\infty\ a.s.$ to ensure that $Y_t^{(j)}<\infty\ a.s.$ for all $t\ge0$ and $j=1,2$. 
Moreover, we assume that for any $t>0$,
\begin{equation}
F_t^{(j)}=\int_{1}^\infty(g^{(j)}{'}(s))^2\sigma_{t-s}^{(j)2}ds<\infty,\ a.s.,\ j=1,2.
\end{equation}
\end{Def}
We denote 
$$R^{(i,j)}(t)=\mathbb{E}[|G_t^{(j)}-G_0^{(i)}|^2],\ R^{(j)}(t)=R^{(j,j)}(t)\ \text{and}\ \tau_n^{(j)}=\sqrt{R^{(j)}\left(\frac{1}{n}\right)},\ i,j=1,2,\ n\ge1.$$
The cross-correlations are given by
$$r_{a,b}^{(n)}(j-i):=\mathbb{E}\left[\frac{\Delta_i^nG^{(a)}}{\tau_n^{(a)}}\frac{\Delta_j^nG^{(b)}}{\tau_n^{(b)}}\right].$$
For function $p(x)=max\{x,0\}=x1_{\{x\ge0\}}$, the realised semicovariance for $Y$ is defined as 
$$V(Y,p)_t^n:=\frac{1}{n}\sum_{i=1}^{[nt]}p\left(\frac{\Delta_i^nY^{(1)}}{\tau_n^{(1)}}\right)p\left(\frac{\Delta_i^nY^{(2)}}{\tau_n^{(2)}}\right).$$
Analogously, the realised semicovariance for $G$ is defined as 
$$V(G,p)_t^n:=\frac{1}{n}\sum_{i=1}^{[nt]}p\left(\frac{\Delta_i^nG^{(1)}}{\tau_n^{(1)}}\right)p\left(\frac{\Delta_i^nG^{(2)}}{\tau_n^{(2)}}\right).$$
Next, we introduce some notations for the bivariate setting. We consider Gaussian vectors\\
$$\bm{X}_i^n:=(X_i^{n(1)},X_i^{n(2)})=\left(\frac{\Delta_i^nG^{(1)}}{\tau_n^{(1)}},\frac{\Delta_i^nG^{(2)}}{\tau_n^{(2)}}\right),\ i\in\mathbb{Z}.$$
Since ${X}^{n(i)}_j,\ i=1,2,\ j=1,2,...,[nt]$ can be regarded as a subset of an isonormal Gaussian process $\{W(u):u\in\mathcal{H}\}$ where $\mathcal{H}$ is an Hilbert space, we can always assume that 
$$X_k^{n(j)}=W(u_{k,j}^n)\quad\text{ and }\quad\langle u_{k,j}^n,u_{k',j'}^n\rangle_{\mathcal{H}}=r_{j,j'}^{(n)}(k'-k),$$
where $j,j'\in\{1,2\}$, $k,k'\in \mathbb{N},\ u_{k,j}^n,u_{k',j'}^n\in\mathcal{H}$ and $r_{j,j'}^{(n)}(k'-k)$ we already defined before.\\
Next, we define our bivariate function $h:\mathbb{R}^2\rightarrow\mathbb{R}^+\cup\{0\}$ as
$$h(x,y):=xy1_{\{x\ge0\}}1_{\{y\ge0\}}.$$
The rescaled sums of the realised semicovariance can be defined as
\begin{equation}
S_n(t):=\frac{1}{\sqrt{n}}\sum_{k=1}^{[nt]}(h(\bm{X}_k^n)-\mathbb{E}[h(\bm{X}_k^n)]),\ n\ge1.
\end{equation}

For simplicity, some results will be established only for the case $t=1$, and we write $S_n$ as a shorthand for $S_n(1)$.\\
By the Hermite expansion of the function $h$, we have
\begin{equation}
S_n=\sum_{m=1}^\infty I_m(g_m^n),\ g_m^n\in\mathcal{H}^{\odot m},
\end{equation}
where $g_m^n$ has the form
\begin{equation}
g_m^n=\frac{1}{\sqrt{n}}\sum_{k=1}^n\sum_{t\in\{1,2\}^m}b_t^nu_{k,t_1}^n\otimes\cdots\otimes u_{k,t_m}^n,
\end{equation}
where $b_t^n$ are certain coefficients such that $t\mapsto b_t^n$ is symmetric on $\{1,2\}^m$.\\

For the multivariate case, we only focus on the central limit theorem of the realised semicovariance. Thus we state here all of the assumptions we need. They can be viewed as the analogues of the assumptions for the univariate case. \\
For $i,j\in\{1,2\}$, we write $\rho_{i,j}=\rho$ for $i\neq j$ and $\rho_{i,j}=1$ for $i=j$. Since 
\begin{equation}\label{multivarigram}
R^{(i,j)}(t)=C_{i,j}+2\rho_{i,j}\int_0^\infty(g^{(j)}(x)-g^{(j)}(x+t))g^{(i)}(x)dx,
\end{equation}
where $C_{i,j}=\|g^{(i)}\|^2_{L^2}+\|g^{(j)}\|^2_{L^2}-2\rho_{i,j}\int_0^\infty g^{(i)}(x)g^{(j)}(x)dx$, we can formulate our assumptions as follows.\\
\begin{asm}\label{multislowlyvarying}
There exist slowly varying (at $0$) functions $L_0^{(i,j)}(t)$ and $L_2^{(i,j)}(t)$ which are continuous on $(0,\infty)$ such that
\begin{equation}
R^{(i,j)}(t)=C_{i,j}+\rho_{i,j}t^{\delta^{(i)}+\delta^{(j)}+1}L_0^{(i,j)}(t),\ i,j=1,2,
\end{equation}
and
\begin{equation}
R^{(i,j)}{''}(t)=\rho_{i,j}t^{\delta^{(i)}+\delta^{(j)}-1}L_2^{(i,j)}(t),\ i,j=1,2,
\end{equation}
where $\delta^{(i)},\delta^{(j)}\in(-\frac{1}{2},0)$. Furthermore, there exist $b\in(0,1)$ such that
\begin{equation}
\limsup_{x\rightarrow0^+}\sup_{y\in[x,x^b]}\left|\frac{L_2^{(i,j)}(y)}{L_0^{(i,j)}(x)}\right|<\infty, i,j=1,2.
\end{equation}
\end{asm}
\begin{asm}\label{multivolatility}
The volatility processes $\sigma_t^{(i)}$ is non-negative and $\eta^{(i)}$-H\"older continuous with $\eta^{(i)}\in(\frac{1}{2},1)$, $i=1,2$.
\end{asm}
\begin{asm}\label{multikernelmeasure}
There exist a constant $\lambda<-1$ such that for any $\epsilon_n=O(n^{-\kappa}),\ \kappa\in(0,1)$, we have
$$\pi_n^{(i)}((\epsilon_n,\infty))=O(n^{\lambda(1-\kappa)}),\ i=1,2,$$
where
$$\pi_n^{(i)}(A)=\frac{\int_A(g^{(i)}(x-1/n)-g^{(i)}(x))^2dx}{\int_0^\infty(g^{(i)}(x-1/n)-g(x))^2dx},\ A\in\mathcal{B}(\mathbb{R}).$$
\end{asm}

\subsection{The central limit theorem for the realised semicovariance}
As in the univariate case, the process
$$S_n(t)=\frac{1}{\sqrt{n}}\sum_{k=1}^{[nt]}(h(\bm{X}_k^n)-\mathbb{E}[h(\bm{X}_k^n)]),\ n\ge1,$$
converges in law to a scaled Brownian motion $\sqrt\beta B$ in the space $\mathcal{D}([0,T])$.\\
To this end, we still need two steps: 
(1) Prove convergence of the finite-dimensional distributions; (2) Prove the tightness of the process.\\
The first step has already been done in \cite{NP2011}, whose result contains our realised semicovariance for the Gaussian cores as a special case. However, the tightness is more difficult to prove. As we know, if we can prove that
 $$\|S_n(t)-S_n(s)\|_{L^p(\Omega)}\le c\left(\frac{[nt]-[ns]}{n}\right)^{\frac{1}{2}},\ 0\le s\le t\le T,$$
for some $p>2$, the tightness will hold. Thanks to Lemma 1 in \cite{BS2013}, we are able to control the $p$-th moment of the increment of the process $S_n(t)$. More specifically, we have the following result.\\

\begin{lem}[\cite{BS2013}, Lemma 1]\label{lemma1}
If $(\bm{X}_1,...,\bm{X}_n)$ is a $\epsilon$-standard Gaussian vector ($\epsilon$-standard means that $|\mathbb{E}X_t^{(u)}X_s^{(v)}|<\epsilon$ holds for any  $1\le u,v\le\nu$ and $t\neq s$), $\bm{X}_t=(X_t^{(1)},...,X_t^{(\nu)})\in\mathbb{R}^\nu,\nu\ge1,$ and $f_{j,t,n}\in L^2(\bm{X}),1\le j\le p,p\ge2, 1\le t\le n$. For given integers $m\ge1, 0\le \alpha\le p,n\ge1$, define
\begin{equation}
Q_n:=\max_{1\le t\le n}\sum_{1\le s\le n,s\neq t}\max_{1\le u,v\le\nu}|\mathbb{E}X_t^{(u)}X_s^{(v)}|^m.
\end{equation}
Assume that $f_{1,t.n},...,f_{\alpha,t,n}$ have a Hermite rank at least equal to $m$ for any $n\ge1,1\le t\le n$, and that $$\epsilon<\frac{1}{\nu p-1}.$$
Then 
\begin{equation}
\sum {}'|\mathbb{E}[f_{1,t_1,n}(\bm{X}_{t_1})\cdots f_{p,t_p,n}(\bm{X}_{t_p})]|\le C(\epsilon,p,m,\alpha,\nu)Kn^{p-\frac{\alpha}{2}}Q_n^{\frac{\alpha}{2}},
\end{equation}
where $\sum {}'$ is the sum over all different indices $1\le t_i\le n\ (1\le i \le p),\ t_i\neq t_j\ (i\neq j)$,
and $$K=\prod_{j=1}^p\max_{1\le t\le n}\|f_{j,t,n}\|\text{ with }\|f_{j,t,n}\|^2=\mathbb{E}[f_{j,t,n}^2(\bm{X})].$$
\end{lem} 
Since our stationary vector is 2-dimensional, we let $\nu=2$. For a fixed $n$, define 
\begin{equation}
f_{i,t_j,n}(\bm{X}^n_{t_j}):=h(\bm{X}^n_{t_j})-\mathbb{E}[h(\bm{X}^n_{t_j})],
\end{equation}
where $\bm{X}_{t_j}^n=(X_{t_j}^{n(1)},X_{t_j}^{n(2)})=\left(\frac{\Delta_{t_j}^nG^{(1)}}{\tau_n^{(1)}},\frac{\Delta_{t_j}^nG^{(2)}}{\tau_n^{(2)}}\right)$. Thus each $f_{i,t_j,n}$ has Hermite rank 1. When $p=4$ and $\alpha=p$, the issue is that our Gaussian vector usually fails to be $\epsilon$-standard with $\epsilon<\frac{1}{7}$. We can overcome this issue by using the so-called ‘decimation technique', which was introduced in \cite{BS2018}. Then we have the functional convergence theorem of our Gaussian cores.\\
\begin{thm}\label{multiGcore} When Assumption \ref{multislowlyvarying} holds, we have
$$\{S_n(t)\}_{0\le t\le T}\overset{\mathcal{L}}{\rightarrow}\sqrt{\beta}\{B_t\}_{0\le t\le T},\ n\rightarrow\infty,$$
in $\mathcal{D}([0,T])$, where $B_t$ is a Brownian motion, $\beta$ is a constant depending on the Hermite coefficients and the limit of cross-correlations.
\end{thm}
As in the univariate case, we can prove the independence between the limiting Brownian motion and the Gaussian cores.\\
\begin{thm}\label{multiweakconvergence} Let Assumption \ref{multislowlyvarying} hold, we have
$$(G_t^{(1)},G_t^{(2)},S(t))\overset{\mathcal{L}}{\rightarrow}(G_t^{(1)},G_t^{(2)},\sqrt{\beta}B_t),$$
in $\mathcal{D}([0,T])^3$, where $B_t$ is a Brownian motion independent of the bivariate process $(G_t^{(1)}-G_0^{(1)},G_t^{(2)}-G_0^{(2)})$.
\end{thm}
Note that in our settings, we already denote the $\sigma$-algebra generated by the process $(G_t^{(1)}-G_0^{(1)},G_t^{(2)}-G_0^{(2)})$ by $\mathcal{F}$. Hence we can easily get a stable convergence version of Theorem \ref{multiweakconvergence}. By applying the blocking technique again, we will have the key result below.\\
\begin{thm}[Central limit theorem]\label{Bi_CLT1} 
Under Assumption \ref{multislowlyvarying}, \ref{multivolatility} and \ref{multikernelmeasure}, if $\sigma^{(1)}$ and $\sigma^{(2)}$ are $\mathcal{F}$-measurable, we have
\begin{eqnarray}
&&\left(\frac{1}{\sqrt{n}}\sum_{i=1}^{[nt]}p\left(\frac{\Delta_i^nX^{(1)}}{\tau_n^{(1)}}\right)p\left(\frac{\Delta_i^nX^{(2)}}{\tau_n^{(2)}}\right)-\sqrt{n}\mathbb{E}\left[h\left(\frac{\Delta_1^nG^{(1)}}{\tau_n^{(1)}},\frac{\Delta_1^nG^{(2)}}{\tau_n^{(2)}}\right)\right]\int_0^t\sigma_s^{(1)}\sigma_s^{(2)}ds\right)_{0\le t\le T}\\
&&\overset{\mathcal{L}}{\rightarrow}\left(\sqrt{\beta}\int_0^t\sigma_s^{(1)}\sigma_s^{(2)}dB_s\right)_{0\le t\le T}
\end{eqnarray}
in $\mathcal{D}([0,T])$, where $B$ is a Brownian motion independent of $\sigma^{(1)}$ and $\sigma^{(2)}$.
\end{thm}

Finally, the above central limit theorem for $p(x)=x1_{\{x\ge0\}}$ can be generalised to a larger class of functions.
\begin{thm}\label{GBi_CLT}
Let $q\ge1$, $\phi(x)=|x|^qI(x)$ where $I(x)=1,1_{\{x\ge0\}}$ or $1_{\{x\le0\}}$ and define $\Phi(x,y):=\phi(x)\phi(y)$. Assume that Assumptions \ref{multislowlyvarying}, \ref{multivolatility} and \ref{multikernelmeasure} hold. If $\sigma^{(1)}$ and $\sigma^{(2)}$ are $\mathcal{F}$-measurable, we have 
\begin{eqnarray}
&&\left(\frac{1}{\sqrt{n}}\sum_{i=1}^{[nt]}\phi\left(\frac{\Delta_i^nX^{(1)}}{\tau_n^{(1)}}\right)\phi\left(\frac{\Delta_i^nX^{(2)}}{\tau_n^{(2)}}\right)-\sqrt{n}\mathbb{E}\left[\Phi\left(\frac{\Delta_1^nG^{(1)}}{\tau_n^{(1)}},\frac{\Delta_1^nG^{(2)}}{\tau_n^{(2)}}\right)\right]\int_0^t\left(\sigma_s^{(1)}\sigma_s^{(2)}\right)^qds\right)_{0\le t\le T}\\
&&\overset{\mathcal{L}}{\rightarrow}\left(\sqrt{\beta}\int_0^t\left(\sigma_s^{(1)}\sigma_s^{(2)}\right)^qdB_s\right)_{0\le t\le T}
\end{eqnarray}
in $\mathcal{D}([0,T])$, where $B$ is a Brownian motion independent of $\sigma^{(1)}$ and $\sigma^{(2)}$.
\end{thm}

\section{Discussion and outlook}
We have established a CLT for the realised semicovariance of a bivariate BSS process which extends the work by \cite{AA2019} and \cite{RA2019} to a non-linear function of increments and can be seen as a non-semimartingale extension of the work by \cite{TB2020}. Moreover, our generalisation from semicovariances to more general functionals provides a more comprehensive multidimensional theory of (semi-)power covariation of BSS processes. Our work can be extended in various directions: For example, we could investigate whether the asymptotic properties of semicovariances still hold when we add jump terms into our BSS processes. Furthermore, since we have noticed that the calculation of our (normalised) realised semicovariance relies on the scaling factors $\tau_n^{(i)},\ i=1,2$, which are unknown from empirical data, this raises the question of whether it is possible to derive feasible covolatility estimators without using the scaling factors $\tau_n^{(i)}$, for instance, by considering suitable ratio statistics.\\

\section{Proofs}
By applying a standard localisation procedure in \cite{BN2006}, our volatility process $\sigma$ can be assumed to be bounded on compact intervals because it is c\`adl\`ag \citep{BN2011}.
\begin{proof}[Proof of Theorem \ref{weakGcore}]
This is a special case of Theorem 1.1 in \cite{ID2018} by letting $\varphi(x)=f(x)-\frac{1}{2}=x^21_{\{x\ge0\}}-\frac{1}{2}$. By applying their results, we prove the weak convergence of finite-dimensional distributions and the tightness of the sequence $\sqrt{n}(V(G,f)_t^n-\frac{1}{2}t)$.
\end{proof}

\begin{proof}[Proof of Theorem \ref{multiclt}]
Let $(c_l,b_l],\ l=1,...,e,$ be disjoint intervals contained in $[0,T]$ and define
$$Z_n^l=\sum_{k=1}^\infty I_k(f_k^{(n,l)}),$$
where $$f_k^{(n,l)}=\frac{a_k}{\sqrt{n}}\sum_{i=[nc_l]+1}^{[nb_l]}\left(\frac{\Delta_i^nG}{\tau_n}\right)^{\otimes k},$$
and 
$$G_n^l=\tau_n\sum_{i=[nc_l]+1}^{[nb_l]}\frac{\Delta_i^nG}{\tau_n},$$
and $\{a_k\}_{k\ge1}$ are Hermitte coefficients of function $g(x)=x^21_{\{x\ge0\}}-\frac{1}{2}$.\\
It suffices to show that
$$(G_n^l,Z_n^l)_{1\le l\le e}\overset{\mathcal{L}}{\rightarrow}(G_{b_l}-G_{c_l},\sqrt{\beta}(B_{b_l}-B_{c_l}))_{1\le l\le e}.$$
For the first component, it is obvious that 
$$G_n^l\overset{a.s.}\rightarrow G_{b_l}-G_{c_l},\ n\rightarrow\infty,$$
and the convergence of the second component follows directly from {Theorem \ref{weakGcore}}.\\
It remains to look at the covariance between the two components.\\
For fixed $l$, by Lemma 1 in \cite{BN2009}, we have $\mathbb{E}[(Z_n^l)^2]\le C_1$
and $\mathbb{E}[(G_n^l)^2]\le C_2$,
then $$|\mathbb{E}[G_n^lZ_n^l]|\le\sqrt{C_1C_2}.$$
Thus
\begin{eqnarray*}
\mathbb{E}[G_n^lZ_n^l]&=&\mathbb{E}\left[Z_n^l\cdot\left(\sum_{j=[nc_l]+1}^{[nb_l]}\Delta_j^nG\right)\right]\\
&=&\mathbb{E}\left[\left(\sum_{k=1}^\infty\frac{a_k}{\sqrt{n}}\sum_{i=[nc_l]+1}^{[nb_l]}H_k\left(\frac{\Delta_i^nG}{\tau_n}\right)\right)\cdot \left(\sum_{j=[nc_l]+1}^{[nb_l]}\Delta_j^nG\right)\right]\\
&\overset{BCT}{=}&\sum_{k=1}^\infty\frac{a_k}{\sqrt{n}}\sum_{i=[nc_l]+1}^{[nb_l]}\mathbb{E}\left[H_k\left(\frac{\Delta_i^nG}{\tau_n}\right)\cdot \left(\sum_{j=[nc_l]+1}^{[nb_l]}\Delta_j^nG\right)\right]\\
&\overset{a_1=\sqrt{\frac{2}{\pi}}}{=}&\frac{a_1}{\sqrt{n}}\sum_{i=[nc_l]+1}^{[nb_l]}\mathbb{E}\left[H_1\left(\frac{\Delta_i^nG}{\tau_n}\right)\cdot \left(\sum_{j=[nc_l]+1}^{[nb_l]}\Delta_j^nG\right)\right]\\
&=&\frac{a_1\tau_n}{\sqrt{n}}\sum_{i=[nc_l]+1}^{[nb_l]}\sum_{j=[nc_l]+11}^{[nb_l]}\mathbb{E}\left[\frac{\Delta_i^nG}{\tau_n}\frac{\Delta_j^nG}{\tau_n}\right]\\
&=&\frac{a_1\tau_n}{\sqrt{n}}\sum_{i=[nc_l]+1}^{[nb_l]}\sum_{j=[nc_l]+1}^{[nb_l]}r_n(|i-j|).\\
\end{eqnarray*}
Then for the equation above, w.l.o.g., we let $c_l=0,b_l=1$, then we have 
\begin{eqnarray}
\mathbb{E}[G_n^lZ_n^l]=\frac{a_1\tau_n}{\sqrt{n}}\sum_{i=[nc_l]+1}^{[nb_l]}\sum_{j=[nc_l]+1}^{[nb_l]}r_n(|i-j|)&=&\frac{a_1\tau_n}{\sqrt{n}}\sum_{i=1}^{n}\sum_{j=1}^{n}r_n(|i-j|)\label{gsumm}\\
&=&\frac{a_1\tau_n}{\sqrt{n}}(n+2\sum_{i=1}^n(n-i)r_n(i)).\label{summ}
\end{eqnarray}
Recall 
$$R(t)=\mathbb{E}[(G_t-G_0)^2],\ t\ge0,\ \ r_n(i)=\frac{R(\frac{i+1}{n})-2R(\frac{i}{n})+R(\frac{i-1}{n})}{2R(\frac{1}{n})},\ i\in\mathbb{N},$$
then we have
\begin{eqnarray*}
\mathbb{E}[G_n^lZ_n^l]&=&\frac{a_1\tau_n}{\sqrt{n}}(n+2\sum_{i=1}^n(n-i)r_n(i))\\
&=&\frac{a_1\tau_n}{\sqrt{n}}\left(n+2\sum_{i=1}^n(n-i)\frac{R(\frac{i+1}{n})-2R(\frac{i}{n})+R(\frac{i-1}{n})}{2R(\frac{1}{n})}\right)\\
&=&\frac{a_1\tau_n}{\sqrt{n}}\left(n+\frac{R(1)-nR(\frac{1}{n})}{R(\frac{1}{n})}\right)\\
&=&\frac{a_1R(1)}{\sqrt{nR(\frac{1}{n})}}.
\end{eqnarray*}
Then the limit depends on the behaviour of the function $R$ near $0$. Since
$$R(t)=t^{2\alpha+1} L_0(t),\ \text{for }\alpha<0,$$
it follows that
$$\frac{R(t)}{t}=t^{2\alpha}L_0(t)\rightarrow\infty,\ t\rightarrow0,$$
(since $\alpha<0$). Then
$$\lim_{n\rightarrow\infty}\mathbb{E}[G_n^lZ_n^l]=0.$$
The central limit theorem for $Z_n^l$ is basically derived from the fourth moment theorem and its generalisation Theorem \ref{GFMT}. In order to apply Theorem \ref{MGFMT} to the multivariate random variable $\left(G_n^l,Z_n^l\right)$ with chaos representation $\left(G_n^l,Z_n^l\right)=\left(I_1(g_1^{(n,l)}),\sum_{k=1}^\infty I_k(f_k^{(n,l)})\right)$, where $g_1^{(n,l)}=\tau_n\sum_{i=[nc_l]+1}^{[nb_l]}u_i^{(n)}$ for $u_i^{(n)}\in\mathcal{H}$ and $\|u_i^{(n)}\|_\mathcal{H}=1$, it only remains to verify condition (ii) for the terms including $u_i^{(n)}$. Notice that
$$\lim_{n\rightarrow\infty}\mathbb{E}[I_1(g_1^{(n,l)})I_1(g_1^{(n,l)})]=R(b_l-c_l),$$
and
$$\lim_{n\rightarrow\infty}\mathbb{E}[I_1(g_1^{(n,l)})I_1(f_1^{(n,l)})]=0$$
by the above discussion, we have that $\left(G_n^l,Z_n^l\right)\xrightarrow[n\rightarrow\infty]{\mathcal{L}}(G_{b_l}-G_{c_l},\sqrt{\beta}(B_{b_l}-B_{c_l})\sim N(0,\Sigma)$, where $\Sigma$ is a $2\times 2$-matrix with $0$ off-diagonal elements and $\beta$ is a constant given by Theorem \ref{weakGcore}. Since $G_{b_l}-G_{c_l}$ and $\sqrt{\beta}(B_{b_l}-B_{c_l})$ are jointly Gaussian and uncorrelated, they are mutually independent. Subsequently, we have the convergence of finite-dimensional distributions
$$(G_n^l,Z_n^l)_{1\le l\le e}\overset{\mathcal{L}}\rightarrow(G_{b_l}-G_{c_l},\sqrt{\beta}(B_{b_l}-B_{c_l})_{1\le l\le e}.$$
Since the tightness of each component has been proved in the previous theorem, the tightness of the bivariate process follows. Finally, we have 
$$(G_t,\sqrt{n}(V(G,f)_t^n-\frac{1}{2}t))\overset{\mathcal{L}}{\rightarrow}(G_t,\sqrt\beta B_t).$$
\end{proof}

\begin{proof}[Proof of Theorem \ref{clt}]
This theorem is a direct consequence of the Theorem \ref{Bi_CLT1} since the realised semivariance is a special case of the realised semicovariance, and the proofs will be identical by omitting superscripts.
\end{proof}

\begin{proof}[Proof of Theorem \ref{multiGcore}]
Let $$\theta^{(n)}(j)=\max_{1\le i,l\le2}|r_{i,l}^{(n)}(j)|,$$
$$K=\inf_{k\in \mathbb{N}}\{\theta^{(n)}(j)\le\frac{1}{2},\forall |j|\ge k\},$$
$$\theta^{(n)}=\sum_{j\in\mathbb{Z}}\theta^{(n)}(j),$$
$$\gamma_{n,m,e}=\sqrt{2\theta^{(n)}n^{-1}\sum_{|j|\le n}\theta^{(n)}(j)^e\sum_{|j|\le n}\theta^{(n)}(j)^{m-e}}.$$
By Theorem 3.1 in \cite{AA2019}, Assumption \ref{multislowlyvarying} leads to 
$$\sup_{n\ge1}\theta^{(n)}<\infty.$$
It is easy to verify that the conditions in Theorem 2.2 of \cite{NP2011} hold.\\
Let $(c_l,b_l],\ l=1,...,e$ be disjoint intervals contained in $[0,T]$. Define
$$S_n^l:=\frac{1}{\sqrt{n}}\sum_{i=[nc_l]+1}^{[nb_l]}(h(X_i^n)-\mathbb{E}[h(X_i^n)])=\sum_{k=1}^\infty I_k(f_k^{(n,l)}),$$
where $$f_k^{(n,l)}=\frac{1}{\sqrt{n}}\sum_{i=[nc_l]+1}^{[nb_l]}\sum_{t\in\{1,2\}^k}b_t^nu_{i,t_1}^n\otimes\cdots\otimes u_{i,t_k}^n.$$ For the finite-dimensional convergence, it remains to prove that for any $1\le l_1\neq l_2\le e$, 
$$\lim_{n\rightarrow\infty}\langle f_k^{(n,l_1)},f_k^{(n,l_2)}\rangle_{\mathbb{H}^{\otimes k}}=0.$$
For $l_1<l_2$, we have
\begin{eqnarray*}
|\langle f_k^{(n,l_1)},f_k^{(n,l_2)}\rangle_{\mathbb{H}^{\otimes k}}|&=&\frac{1}{n}\left|\langle\sum_{i=[nc_{l_1}]+1}^{[nb_{l_1}]}\sum_{t\in\{1,2\}^k}b_tu_{i,t_1}^n\otimes\cdots\otimes u_{i,t_k}^n,\sum_{i=[nc_{l_2}]+1}^{[nb_{l_2}]}\sum_{t\in\{1,2\}^k}b_tu_{i,t_1}^n\otimes\cdots\otimes u_{i,t_k}^n\rangle_{\mathbb{H}^{\otimes k}}\right|\\
&\le&\frac{(\sum_{t\in\{1,2\}^k}|b_t^n|)^2}{n}\sum_{i=[nc_{l_1}]+1}^{[nb_{l_1}]}\sum_{j=[nc_{l_2}]+1}^{[nb_{l_2}]}\theta^{(n)}(j-i)^k\\
&\le&\frac{(\sum_{t\in\{1,2\}^k}|b_t^n|)^2}{n}\sum_{i=[nc_{l_1}]+1}^{[nb_{l_1}]}\sum_{j=[nc_{l_2}]+1}^{[nb_{l_2}]}\theta^{(n)}(j-i).\\
\end{eqnarray*}
Assume w.l.o.g. that $c_{l_1}=0,b_{l_1}=c_{l_2}=1, b_{l_2}=2$. Then 
$$|\langle f_k^{(n,l_1)},f_k^{(n,l_2)}\rangle_{\mathbb{H}^{\otimes k}}|\le (\sum_{t\in\{1,2\}^k}|b_t^n|)^2\left(\frac{1}{n}\sum_{j=1}^nj\theta^{(n)}(j)+\sum_{j=1}^{n-1}\theta^{(n)}(n+j)\right)\rightarrow0,$$
since $\sup_{n\ge1}\sum_{j=1}^\infty\theta^{(n)}(j)<\infty$ by our assumption. (We can find a dominant sequence such that $1\ge\theta(j)\ge\sup_{n\ge1}\theta^{(n)}(j)$ and rewrite our assumption as $\theta:=\sum_{j\in\mathbb{Z}}\theta(j)<\infty$.)\\ 
Hence the convergence of finite-dimensional distributions holds. Next, we prove the tightness.\\
We initially ignore the issue about the $\epsilon$-standard property and apply Lemma \ref{lemma1} to prove the tightness of $\{S_n(t)\}_{0\le t\le T}$.\\
For a fixed $n$ and $0\le s<t\le T$, let $N=[nt]-[ns]$ and define 
$$f_{j,i,N}(\bm{X}_{i}):=h(\bm{X}_{i})-\mathbb{E}[h(\bm{X}_{i})],\ 1\le j\le 4,\ [ns]+1\le i\le [nt].$$
Thus each $f_{j,i,N}$ has a Hermite rank 1, i.e., $m=1$.
\begin{eqnarray*}
\mathbb{E}[|S_n(t)-S_n(s)|^4]&=&\frac{1}{n^2}\mathbb{E}\left[\left(\sum_{i=[ns]+1}^{[nt]}h(\bm{X}_{i})-\mathbb{E}[h(\bm{X}_{i})]\right)^4\right]\\
&=& \frac{1}{n^2}\mathbb{E}\left[\left(\sum_{i=[ns]+1}^{[nt]}f_{j,i,N}(\bm{X}_{i})\right)^4\right]\\
&\le&\frac{1}{n^2}(\Sigma_4+\Sigma_3+\Sigma_{2,1}+\Sigma_{2,2}),
\end{eqnarray*}
where 
$$\Sigma_4=\sum{}'\mathbb{E}|[f_{1,t_1,N}(\bm{X}_{t_1})\cdots f_{4,t_4,N}(\bm{X}_{t_4})]|,$$
$$\Sigma_3=\sum{}'\mathbb{E}|[f_{1,t_1,N}(\bm{X}_{t_1})f_{2,t_2,N}(\bm{X}_{t_2})f_{3,t_3,N}^2(\bm{X}_{t_3})]|,$$
$$\Sigma_{2,1}=\sum{}'\mathbb{E}|[f_{1,t_1,N}^3(\bm{X}_{t_1})f_{2,t_2,N}(\bm{X}_{t_2})|,$$
$$\Sigma_{2,2}=\sum\mathbb{E}|[f_{1,t_1,N}^2(\bm{X}_{t_1})f_{2,t_2,N}^2(\bm{X}_{t_2})|,$$
where $\sum{}′$ stands for the sum over all different integers $[ns]+1\le t_p\le[nt]$.\\
Since we have proved that all fourth-order moments of $h(\bm{X}_i)$ can be controlled by a constant, we have
$$\Sigma_{2,1}+\Sigma_{2,2}\le C_1\sum{}'1+C_2\sum1\le C'([nt]-[ns])^2.$$
Next, we seek to control $$Q_n=\max_{[ns]+1\le n_1\le [nt]}\sum_{[ns]+1\le n_2\le [nt],n_2\neq n_1}\max_{1\le u,v\le2}|\mathbb{E}X_{n_2}^{(u)}X_{n_1}^{(v)}|.$$
We use the notation we have defined before, $\theta^{(n)}(n_2-n_1)=\max_{1\le u,v\le2}|\mathbb{E}X_{n_2}^{(u)}X_{n_1}^{(v)}|$ and $\theta^{(n)}=\sum_{j\in\mathbb{Z}}\theta^{(n)}(j)$.\\
It is easy to see that 
$$Q_n\le \theta^{(n)}.$$
Recall our assumption that $\sup_{n\ge1}\theta^{(n)}<\infty$. Hence $Q_n$ can be controlled by a constant, i.e., 
$$Q_n\le C''.$$
Then we use Lemma 3.1 for $p=4,\alpha=4$, we have
$$\Sigma_4=\sum{}'\mathbb{E}|[f_{1,t_1,N}(\bm{X}_{t_1})\cdots f_{4,t_4,N}(\bm{X}_{t_4})]|\le C(\epsilon,p,m,\alpha,\nu)KN^2Q_n^2=C_4([nt]-[ns])^2,$$
and for $p=3,\alpha=2$, we have
$$\Sigma_3=\sum{}'\mathbb{E}|[f_{1,t_1,N}(\bm{X}_{t_1})f_{2,t_2,N}(\bm{X}_{t_2})f_{3,t_3,N}^2(\bm{X}_{t_3})]|\le C(\epsilon,p,m,\alpha,\nu)KN^2Q_n=C_3([nt]-[ns])^2.$$
Altogether, we have
$$\mathbb{E}[|S_n(t)-S_n(s)|^4]\le C\left(\frac{[nt]-[ns]}{n}\right)^2.$$
Finally, we revisit the issue of the $\epsilon$-standard property. We will use the decimation technique introduced in \cite{BS2018} to deal with it.\\	
We divide $[nt]-[ns]$ terms in the sum below
$$S_n(t)-S_n(s)=\frac{1}{\sqrt{n}}\sum_{i=[ns]+1}^{[nt]}h(\bm{X}_{i})-\mathbb{E}[h(\bm{X}_{i})]$$
into $l$ groups where the indices will differ by at least $l$.\\
Define $$T_{n,l}(j)=\frac{1}{\sqrt{n}}\sum_{[ns]<k\le [nt]:k=j(\text{mod }l)}h(\bm{X}_{k})-\mathbb{E}[h(\bm{X}_{k})],$$
then we have
\begin{eqnarray*}
\mathbb{E}[|S_n(t)-S_n(s)|^4]&=&\frac{1}{n^2}\mathbb{E}\left[\left(\sum_{i=[ns]+1}^{[nt]}h(\bm{X}_{i})-\mathbb{E}[h(\bm{X}_{i})]\right)^4\right]\\
&=&\mathbb{E}\left[\left(\sum_{j=0}^{l-1}T_{n,l}(j)\right)^4\right]\\
&\le& l^4\max_{0\le j<l}\mathbb{E}\left[T_{n,l}(j)^4\right].
\end{eqnarray*}
Next, we choose a finite $l$ such that the $\epsilon-$standard condition will be satisfied. Since we have $r_{a,b}^{(k)}\le C(k-1)^{\delta^{(a)}+\delta^{(b)}-1+\epsilon'}$ where $0<\epsilon'<1-\delta^{(a)}-\delta^{(b)}$ is a constant and $C$ is a constant independent of $n$. Then for sufficiently large $l$, we will have $|r_{a,b}^n(k)|\le\epsilon<\frac{1}{\nu p-1}=\frac{1}{7}$ for all $k\ge l$. We shall fix this $l$.\\
Thanks to the previous partition, the difference of indices in each sum $T_{n,l}(j)$ will be at least $l$. Thus for random vectors in each group, they satisfy the $\epsilon-$standard condition.\\
For $\mathbb{E}[T_{n,l}(j)^4]$, we use similar arguments as before and deduce
$$\mathbb{E}[T_{n,l}(j)^4]\le\frac{1}{l^2}C(j)\left(\frac{[nt]-[ns]}{n}\right)^2\ j=0,1,...,l-1,$$
where $C(j)$ are positive constants.\\
Then 
$$\mathbb{E}[|S_n(t)-S_n(s)|^4]\le l^4\max_{0\le j<l}\mathbb{E}[T_{n,l}(j)^4]\le C\left(\frac{[nt]-[ns]}{n}\right)^2,$$
where $C=l^2\max_{0\le j<l}C(j)$.\\
This completes the proof of the tightness.
\end{proof}

\begin{proof}[Proof of Theorem \ref{multiweakconvergence}]
We need to analyse the dependence structure between the three components. \\
Let $(c_l,b_l],\ l=1,...,e$ be disjoint intervals contained in $[0,T]$ and define
$$Z_n^l:=S_n(b_l)-S_n(c_l)$$
and $$G_n^{(i)l}=\tau_n^{(i)}\sum_{j=[nc_l]+1}^{[nb_l]}\frac{\Delta_j^nG^{(i)}}{\tau_n^{(i)}},\ i=1,2.$$
We denote $e_{n,j}^{(i)}:=\frac{\Delta_j^nG^{(i)}}{\tau_n^{(i)}}$. Recall the Hermite decomposition we made in Section 3.1, the term of the first Hermite component of $S_n(b_l)-S_n(c_l)$ is
$$I_1(f_1^{(n)})=\frac{A_n}{\sqrt{n}}\sum_{j=[nc_l]+1}^{[nb_l]}(e_{n,j}^{(1)}+e_{n,j}^{(2)}),$$
where $A_n$ is a generic coefficient.\\
Then 
\begin{eqnarray*}
\mathbb{E}[Z_n^lG_n^{(1)l}]&=&\frac{A_n\tau_n^{(1)}}{\sqrt{n}}\sum_{j=[nc_l]+1}^{[nb_l]}\sum_{i=[nc_l]+1}^{[nb_l]}\mathbb{E}\left[\left(\frac{\Delta_j^nG^{(1)}}{\tau_n^{(1)}}+\frac{\Delta_j^nG^{(2)}}{\tau_n^{(2)}}\right)\left(\frac{\Delta_i^nG^{(1)}}{\tau_n^{(1)}}\right)\right]\\
&=&\frac{A_n\tau_n^{(1)}}{\sqrt{n}}\sum_{j=[nc_l]+1}^{[nb_l]}\sum_{i=[nc_l]+1}^{[nb_l]}\mathbb{E}
\left[\frac{\Delta_j^nG^{(1)}}{\tau_n^{(1)}}\frac{\Delta_i^nG^{(1)}}{\tau_n^{(1)}}\right]\\
&&+\frac{A_n\tau_n^{(1)}}{\sqrt{n}}\sum_{j=[nc_l]+1}^{[nb_l]}\sum_{i=[nc_l]+1}^{[nb_l]}\mathbb{E}\left[\frac{\Delta_j^nG^{(2)}}{\tau_n^{(2)}}\frac{\Delta_i^nG^{(1)}}{\tau_n^{(1)}}\right].
\end{eqnarray*}
Recall the assumption we made
$$R^{(i,j)}(t):=\mathbb{E}[(G_t^{(j)}-G_0^{(i)})^2]=C_{i,j}+\rho_{i,j}t^{\delta^{(i)}+\delta^{(j)}+1}L_0^{(i,j)}(t),$$
where $C_{i,j}=0$ if $i=j$ and $\rho_{i,j}=\rho$ if $i\neq j$, where we assumed that the driving Brownian motions $W_t^{(1)}$ and $W_t^{(2)}$ have constant correlation $\rho$.\\
W.l.o.g., we let $b_l=1,\ c_l=0$, then the first term above is
$$\frac{A_n\tau_n^{(1)}}{\sqrt{n}}\sum_{j=[nc_l]+1}^{[nb_l]}\sum_{i=[nc_l]+1}^{[nb_l]}\frac{\Delta_j^nG^{(1)}}{\tau_n^{(1)}}\frac{\Delta_i^nG^{(1)}}{\tau_n^{(1)}}=\frac{A_nR^{(1,1)}(1)}{\sqrt{nR^{(1,1)}(\frac{1}{n})}},$$
which will converge to $0$ if $\delta^{(1)}<0$.\\
The second term is
\begin{eqnarray*}
&&\frac{A_n\tau_n^{(1)}}{\sqrt{n}}\sum_{j=[nc_l]+1}^{[nb_l]}\sum_{i=[nc_l]+1}^{[nb_l]}\frac{\Delta_j^nG^{(2)}}{\tau_n^{(2)}}\frac{\Delta_i^nG^{(1)}}{\tau_n^{(1)}}=\frac{A_n\tau_n^{(1)}}{\sqrt{n}}\sum_{j=1}^{n}\sum_{i=1}^{n}r_{1,2}^n(j-i)\\
&=&\frac{A_n\tau_n^{(1)}}{\sqrt{n}}\frac{R^{(1,2)}(1)+R^{(2,1)}(1)+(n-1)(R^{(1,2)}(0)+R^{(2,1)}(0))-n(R^{(1,2)}(\frac{1}{n})+R^{(2,1)}(\frac{1}{n}))}{2\tau_n^{(1)}\tau_n^{(2)}}\\
&&+\frac{2n\mathbb{E}[\Delta_1^nG_n^{(1)}\Delta_1^nG_n^{(2)}]}{2\tau_n^{(1)}\tau_n^{(2)}}\\
&=&\frac{A_n\tau_n^{(1)}}{\sqrt{n}}\frac{R^{(1,2)}(1)+R^{(2,1)}(1)+(n-2)\mathbb{E}[(G_0^{(2)}-G_0^{(1)})^2]-n\mathbb{E}[(G_{\frac{1}{n}}^{(2)}-G_{\frac{1}{n}}^{(1)})^2]}{2\tau_n^{(1)}\tau_n^{(2)}}\\
&=&\frac{A_n(\rho L_0^{(1,2)}(1)+\rho L_0^{(2,1)}(1))}{\sqrt{nR^{(2,2)}(\frac{1}{n})}},
\end{eqnarray*}
which converges to $0$ if $\delta^{(2)}<0$.\\
Thus the independence between the Gaussian cores and the limiting Brownian motion has been proved since both of them are Gaussian.\\
\end{proof}

\begin{proof}[Proof of Theorem \ref{Bi_CLT1}]
We denote $\mu_n:=\mathbb{E}\left[h\left(\frac{\Delta_1^nG^{(1)}}{\tau_n^{(1)}},\frac{\Delta_1^nG^{(2)}}{\tau_n^{(2)}}\right)\right]$.\\
Since $\sigma^{(1)},\ \sigma^{(2)}$ are non-negative, we have
$$p(\sigma_s^{(i)})=\sigma_s^{(i)},$$
and
$$p\left(\sigma_s^{(i)}\frac{\Delta_j^nG^{(i)}}{\tau_n^{(i)}}\right)=\sigma_s^{(i)}p\left(\frac{\Delta_j^nG^{(i)}}{\tau_n^{(i)}}\right),\  i=1,2.$$

For a fixed $n$, let $l\le n$, then we have the following decomposition
\begin{eqnarray*}
&&\frac{1}{\sqrt{n}}\sum_{i=1}^{[nt]}p\left(\frac{\Delta_i^nX^{(1)}}{\tau_n^{(1)}}\right)p\left(\frac{\Delta_i^nX^{(2)}}{\tau_n^{(2)}}\right)-\sqrt{n}\mathbb{E}\left[h\left(\frac{\Delta_1^nG^{(1)}}{\tau_n^{(1)}},\frac{\Delta_1^nG^{(2)}}{\tau_n^{(2)}}\right)\right]\int_0^t\sigma_s^{(1)}\sigma_s^{(2)}ds\\
&=&\frac{1}{\sqrt{n}}\sum_{i=1}^{[nt]}p\left(\frac{\Delta_i^nX^{(1)}}{\tau_n^{(1)}}\right)p\left(\frac{\Delta_i^nX^{(2)}}{\tau_n^{(2)}}\right)-\sqrt{n}\mu_n\int_0^t\sigma_s^{(1)}\sigma_s^{(2)}ds
\\
&=&\frac{1}{\sqrt{n}}\sum_{i=1}^{[nt]}\left(p\left(\frac{\Delta_i^nX^{(1)}}{\tau_n^{(1)}}\right)p\left(\frac{\Delta_i^nX^{(2)}}{\tau_n^{(2)}}\right)-p\left(\sigma^{(1)}_{(i-1)/n}\frac{\Delta_i^nG^{(1)}}{\tau_n^{(1)}}\right)p\left(\sigma^{(2)}_{(i-1)/n}\frac{\Delta_i^nG^{(2)}}{\tau_n^{(2)}}\right)\right)\\
&&+\frac{1}{\sqrt{n}}\sum_{i=1}^{[nt]}p\left(\sigma^{(1)}_{(i-1)/n}\frac{\Delta_i^nG^{(1)}}{\tau_n^{(1)}}\right)p\left(\sigma^{(2)}_{(i-1)/n}\frac{\Delta_i^nG^{(2)}}{\tau_n^{(2)}}\right)\\
&&-\frac{1}{\sqrt{n}}\sum_{j=1}^{[lt]}\sigma^{(1)}_{(j-1)/l}\sigma^{(2)}_{(j-1)/l}\sum_{i\in I_l(j)}p\left(\frac{\Delta_i^nG^{(1)}}{\tau_n^{(1)}}\right)p\left(\frac{\Delta_i^nG^{(2)}}{\tau_n^{(2)}}\right)\\
&&+\frac{\sqrt{n}}{l}\mu_n\sum_{j=1}^{[lt]}\sigma_{(j-1)/l}^{(1)}\sigma_{(j-1)/l}^{(2)}-\frac{1}{\sqrt{n}}\sum_{i=1}^{[nt]}\sigma_{(i-1)/n}^{(1)}\sigma_{(i-1)/n}^{(2)}\\
&&+\frac{1}{\sqrt{n}}\sum_{j=1}^{[lt]}\sigma^{(1)}_{(j-1)/l}\sigma^{(2)}_{(j-1)/l}\sum_{i\in I_l(j)}p\left(\frac{\Delta_i^nG^{(1)}}{\tau_n^{(1)}}\right)p\left(\frac{\Delta_i^nG^{(2)}}{\tau_n^{(2)}}\right)-\frac{\sqrt{n}}{l}\mu_n\sum_{j=1}^{[lt]}\sigma_{(j-1)/l}^{(1)}\sigma_{(j-1)/l}^{(2)}\\
&&+\frac{1}{\sqrt{n}}\sum_{i=1}^{[nt]}\sigma_{(i-1)/n}^{(1)}\sigma_{(i-1)/n}^{(2)}-\sqrt{n}\mu_n\int_0^t\sigma_s^{(1)}\sigma_s^{(2)}ds,\\
\end{eqnarray*}
where $$I_l(j)=\left\{i:\frac{i}{n}\in\left(\frac{j-1}{l},\frac{j}{l}\right]\right\}.$$
Define
\begin{eqnarray*}
A_t^n:&=&\frac{1}{\sqrt{n}}\sum_{i=1}^{[nt]}\left(p\left(\frac{\Delta_i^nX^{(1)}}{\tau_n^{(1)}}\right)p\left(\frac{\Delta_i^nX^{(2)}}{\tau_n^{(2)}}\right)-p\left(\sigma^{(1)}_{(i-1)/n}\frac{\Delta_i^nG^{(1)}}{\tau_n^{(1)}}\right)p\left(\sigma^{(2)}_{(i-1)/n}\frac{\Delta_i^nG^{(2)}}{\tau_n^{(2)}}\right)\right),\\
B_t^{'n,l}:&=&\frac{1}{\sqrt{n}}\sum_{i=1}^{[nt]}p\left(\sigma^{(1)}_{(i-1)/n}\frac{\Delta_i^nG^{(1)}}{\tau_n^{(1)}}\right)p\left(\sigma^{(2)}_{(i-1)/n}\frac{\Delta_i^nG^{(2)}}{\tau_n^{(2)}}\right)\\
&&-\frac{1}{\sqrt{n}}\sum_{j=1}^{[lt]}\sigma^{(1)}_{(j-1)/l}\sigma^{(2)}_{(j-1)/l}\sum_{i\in I_l(j)}p\left(\frac{\Delta_i^nG^{(1)}}{\tau_n^{(1)}}\right)p\left(\frac{\Delta_i^nG^{(2)}}{\tau_n^{(2)}}\right)\\
&=&\frac{1}{\sqrt{n}}\sum_{i=1}^{[nt]}\sigma^{(1)}_{(i-1)/n}\sigma^{(2)}_{(i-1)/n}p\left(\frac{\Delta_i^nG^{(1)}}{\tau_n^{(1)}}\right)p\left(\frac{\Delta_i^nG^{(2)}}{\tau_n^{(2)}}\right)\\
&&-\frac{1}{\sqrt{n}}\sum_{j=1}^{[lt]}\sigma^{(1)}_{(j-1)/l}\sigma^{(2)}_{(j-1)/l}\sum_{i\in I_l(j)}p\left(\frac{\Delta_i^nG^{(1)}}{\tau_n^{(1)}}\right)p\left(\frac{\Delta_i^nG^{(2)}}{\tau_n^{(2)}}\right)\\
B_t^{''n,l}:&=&\frac{\sqrt{n}}{l}\mu_n\sum_{j=1}^{[lt]}\sigma_{(j-1)/l}^{(1)}\sigma_{(j-1)/l}^{(2)}-\frac{1}{\sqrt{n}}\sum_{i=1}^{[nt]}\sigma_{(i-1)/n}^{(1)}\sigma_{(i-1)/n}^{(2)},\\
C_{t}^{n,l}:&=&\frac{1}{\sqrt{n}}\sum_{j=1}^{[lt]}\sigma^{(1)}_{(j-1)/l}\sigma^{(2)}_{(j-1)/l}\sum_{i\in I_l(j)}p\left(\frac{\Delta_i^nG^{(1)}}{\tau_n^{(1)}}\right)p\left(\frac{\Delta_i^nG^{(2)}}{\tau_n^{(2)}}\right)-\frac{\sqrt{n}}{l}\mu_n\sum_{j=1}^{[lt]}\sigma_{(j-1)/l}^{(1)}\sigma_{(j-1)/l}^{(2)},\\
D_t^{n}:&=&\frac{1}{\sqrt{n}}\sum_{i=1}^{[nt]}\sigma_{(i-1)/n}^{(1)}\sigma_{(i-1)/n}^{(2)}-\sqrt{n}\mu_n\int_0^t\sigma_s^{(1)}\sigma_s^{(2)}ds.\\
\end{eqnarray*}
We first prove that $$A_t^n\overset{P}{\rightarrow}0,\ n\rightarrow\infty.$$
For function $p(x)=x1_{\{x\ge0\}}$, we have
$$|p(a)-p(b)|\le|a-b|,$$
$$p(a)p(b)-p(c)p(d)=\frac{1}{2}[(p(a)-p(c))(p(b)+p(d))+(p(b)-p(d))(p(a)+p(c))].$$
Moreover, by Lemma 1 in \cite{BN2011}, we have
$$\mathbb{E}\left[\left|\Delta_i^nX^{(j)}\right|\right]\le C\tau_n^{(j)},\ \mathbb{E}\left[\left|\sigma_{(i-1)/n}^{(j)}\Delta_i^nG^{(j)}\right|\right]\le C\tau_n^{(j)}\ i=1,...,[nt],j=1,2.$$
Then 
\begin{eqnarray*}
\mathbb{E}[|A_t^n|]&=&\frac{1}{\sqrt{n}}\mathbb{E}\left[\left|\sum_{i=1}^{[nt]}\left(p\left(\frac{\Delta_i^nX^{(1)}}{\tau_n^{(1)}}\right)p\left(\frac{\Delta_i^nX^{(2)}}{\tau_n^{(2)}}\right)-p\left(\sigma^{(1)}_{(i-1)/n}\frac{\Delta_i^nG^{(1)}}{\tau_n^{(1)}}\right)p\left(\sigma^{(2)}_{(i-1)/n}\frac{\Delta_i^nG^{(2)}}{\tau_n^{(2)}}\right)\right)\right|\right]\\
&\le&\frac{1}{\sqrt{n}}\sum_{i=1}^{[nt]}\mathbb{E}\left|p\left(\frac{\Delta_i^nX^{(1)}}{\tau_n^{(1)}}\right)p\left(\frac{\Delta_i^nX^{(2)}}{\tau_n^{(2)}}\right)-p\left(\sigma^{(1)}_{(i-1)/n}\frac{\Delta_i^nG^{(1)}}{\tau_n^{(1)}}\right)p\left(\sigma^{(2)}_{(i-1)/n}\frac{\Delta_i^nG^{(2)}}{\tau_n^{(2)}}\right)\right|\\
&=&\frac{1}{2\sqrt{n}}\sum_{i=1}^{[nt]}\mathbb{E}\left|\left(p\left(\frac{\Delta_i^nX^{(1)}}{\tau_n^{(1)}}\right)-p\left(\sigma^{(1)}_{(i-1)/n}\frac{\Delta_i^nG^{(1)}}{\tau_n^{(1)}}\right)\right)\left(p\left(\frac{\Delta_i^nX^{(2)}}{\tau_n^{(2)}}\right)+p\left(\sigma^{(2)}_{(i-1)/n}\frac{\Delta_i^nG^{(2)}}{\tau_n^{(2)}}\right)\right)\right.\\
&&+\left.\left(p\left(\frac{\Delta_i^nX^{(2)}}{\tau_n^{(2)}}\right)-p\left(\sigma^{(2)}_{(i-1)/n}\frac{\Delta_i^nG^{(2)}}{\tau_n^{(2)}}\right)\right)\left(p\left(\frac{\Delta_i^nX^{(1)}}{\tau_n^{(1)}}\right)+p\left(\sigma^{(1)}_{(i-1)/n}\frac{\Delta_i^nG^{(1)}}{\tau_n^{(1)}}\right)\right)\right|\\
&\le&\frac{1}{2\sqrt{n}}\sum_{i=1}^{[nt]}\mathbb{E}\left|\left(p\left(\frac{\Delta_i^nX^{(1)}}{\tau_n^{(1)}}\right)-p\left(\sigma^{(1)}_{(i-1)/n}\frac{\Delta_i^nG^{(1)}}{\tau_n^{(1)}}\right)\right)\left(p\left(\frac{\Delta_i^nX^{(2)}}{\tau_n^{(2)}}\right)+p\left(\sigma^{(2)}_{(i-1)/n}\frac{\Delta_i^nG^{(2)}}{\tau_n^{(2)}}\right)\right)\right|\\
&&+\frac{1}{2\sqrt{n}}\sum_{i=1}^{[nt]}\mathbb{E}\left|\left(p\left(\frac{\Delta_i^nX^{(2)}}{\tau_n^{(2)}}\right)-p\left(\sigma^{(2)}_{(i-1)/n}\frac{\Delta_i^nG^{(2)}}{\tau_n^{(2)}}\right)\right)\left(p\left(\frac{\Delta_i^nX^{(1)}}{\tau_n^{(1)}}\right)
+p\left(\sigma^{(1)}_{(i-1)/n}\frac{\Delta_i^nG^{(1)}}{\tau_n^{(1)}}\right)\right)\right|\\
&\le&\frac{C}{2\sqrt{n}}\sum_{i=1}^{[nt]}\left(\mathbb{E}\left[\left(p\left(\frac{\Delta_i^nX^{(1)}}{\tau_n^{(1)}}\right)-p\left(\sigma^{(1)}_{(i-1)/n}\frac{\Delta_i^nG^{(1)}}{\tau_n^{(1)}}\right)\right)^2\right]\right)^{\frac{1}{2}}\\
&&+\frac{C}{2\sqrt{n}}\sum_{i=1}^{[nt]}\left(\mathbb{E}\left[\left(p\left(\frac{\Delta_i^nX^{(2)}}{\tau_n^{(2)}}\right)-p\left(\sigma^{(2)}_{(i-1)/n}\frac{\Delta_i^nG^{(2)}}{\tau_n^{(2)}}\right)\right)^2\right]\right)^{\frac{1}{2}}\\
&\le&\frac{C}{2\sqrt{n}}\sum_{i=1}^{[nt]}\left(\mathbb{E}\left[\left(\frac{\Delta_i^nX^{(1)}}{\tau_n^{(1)}}-\sigma^{(1)}_{(i-1)/n}\frac{\Delta_i^nG^{(1)}}{\tau_n^{(1)}}\right)^2\right]\right)^{\frac{1}{2}}\\
&&+\frac{C}{2\sqrt{n}}\sum_{i=1}^{[nt]}\left(\mathbb{E}\left[\left(\frac{\Delta_i^nX^{(2)}}{\tau_n^{(2)}}-\sigma^{(2)}_{(i-1)/n}\frac{\Delta_i^nG^{(2)}}{\tau_n^{(2)}}\right)^2\right]\right)^{\frac{1}{2}}.\\
\end{eqnarray*}
Hence we reduce the problem to the univariate case, while the proof of
$$\frac{1}{\sqrt{n}\tau_n^{(1)}}\sum_{i=1}^{[nt]}\left(\mathbb{E}\left[\left|\Delta_i^nX^{(1)}-\sigma_{(i-1)/n}^{(1)}\Delta_i^nG^{(1)}\right|^2\right]\right)^{\frac{1}{2}}\rightarrow0,\ n\rightarrow\infty,$$
$$\frac{1}{\sqrt{n}\tau_n^{(2)}}\sum_{i=1}^{[nt]}\left(\mathbb{E}\left[\left|\Delta_i^nX^{(2)}-\sigma_{(i-1)/n}^{(2)}\Delta_i^nG^{(2)}\right|^2\right]\right)^{\frac{1}{2}}\rightarrow0,\ n\rightarrow\infty,$$
can be found in the proof of Theorem 5 in \cite{BN2011}.\\
Next, we prove that 
$$\mathbb{P}-\lim_{l\rightarrow\infty}\limsup_{n\rightarrow\infty}\sup_{t\in[0,T]}|B_t^{'n,l}+B_t^{''n,l}|=0.$$
We denote
$$B_t^{n,l}:=\sum_{j=1}^{[lt]}\sum_{i\in I_l(j)}\left(\sigma_{(i-1)/n}^{(1)}\sigma_{(i-1)/n}^{(2)}-\sigma_{(j-1)/l}^{(1)}\sigma_{(j-1)/l}^{(2)}\right)\times\frac{1}{\sqrt{n}}\left(p\left(\frac{\Delta_i^nG^{(1)}}{\tau_n^{(1)}}\right)p\left(\frac{\Delta_i^nG^{(2)}}{\tau_n^{(2)}}\right)-\mu_n\right).$$
Then we have, for a fixed $l$,
$$|B_t^{'n,l}+B_t^{''n,l}-B_t^{n,l}|\le\left|\frac{\mu_n}{\sqrt{n}}\sum_{j=1}^{[lt]}\sigma_{(j-1)/n}^{(1)}\sigma_{(j-1)/n}^{(2)}\right|\overset{a.s.}{\rightarrow}0,\ n\rightarrow\infty.$$
Since the volatility processes $\sigma^{(i)}$ are $\eta^{(i)}$-H\"older continuous and bounded on compact interval for $i=1,2$, we have that $\sigma_{(i-1)/n}^{(1)}\sigma_{(i-1)/n}^{(2)}-\sigma_{(j-1)/l}^{(1)}\sigma_{(j-1)/l}^{(2)}$ is $\eta^{(1)}\wedge\eta^{(2)}$-H\"older continuous where $a\wedge b$ means $min\{a,b\}$, then by the convergence (4) in \cite{CN2014}, we have
$$\mathbb{P}-\lim_{l\rightarrow\infty}\limsup_{n\rightarrow\infty}\sup_{t\in[0,T]}|B_t^{n,l}|=0.$$
Hence $$\mathbb{P}-\lim_{l\rightarrow\infty}\limsup_{n\rightarrow\infty}\sup_{t\in[0,T]}|B_t^{'n,l}+B_t^{''n,l}|=0.$$
For $C_t^{n,l}$, we have
\begin{eqnarray*}
C_t^{n,l}&=&\frac{1}{\sqrt{n}}\sum_{j=1}^{[lt]}\sigma^{(1)}_{(j-1)/l}\sigma^{(2)}_{(j-1)/l}\sum_{i\in I_l(j)}p\left(\frac{\Delta_j^nG^{(1)}}{\tau_n^{(1)}}\right)p\left(\frac{\Delta_j^nG^{(2)}}{\tau_n^{(2)}}\right)-\frac{\sqrt{n}}{l}\mu_n\sum_{j=1}^{[lt]}\sigma_{(j-1)/l}^{(1)}\sigma_{(j-1)/l}^{(2)}\\
&=&\frac{1}{\sqrt{n}}\sum_{j=1}^{[lt]}\sigma^{(1)}_{(j-1)/l}\sigma^{(2)}_{(j-1)/l}\sum_{i\in I_l(j)}\left(p\left(\frac{\Delta_j^nG^{(1)}}{\tau_n^{(1)}}\right)p\left(\frac{\Delta_j^nG^{(2)}}{\tau_n^{(2)}}\right)-\mu_n\right)\\
&&+\left(\frac{\#I_l(j)}{\sqrt{n}}-\frac{\sqrt{n}}{l}\right)\mu_n\sum_{j=1}^{[lt]}\sigma_{(j-1)/l}^{(1)}\sigma_{(j-1)/l}^{(2)}.
\end{eqnarray*}
Since $\#I_l(j)=\left[\frac{n}{l}\right]$ or $\left[\frac{n}{l}\right]+1$, we have
$$\left|\left(\frac{\#I_l(j)}{\sqrt{n}}-\frac{\sqrt{n}}{l}\right)\mu_n\sum_{j=1}^{[lt]}\sigma_{(j-1)/l}^{(1)}\sigma_{(j-1)/l}^{(2)}\right|\le\left|\frac{\mu_n}{\sqrt{n}}\sum_{j=1}^{[lt]}\sigma_{(j-1)/l}^{(1)}\sigma_{(j-1)/l}^{(2)}\right|\rightarrow0,n\rightarrow\infty.$$
If we denote by $\mathcal{G}$ the $\sigma$-algebra generated by the process $\{(G_t^{(1)}-G_0^{(1)},G_t^{(2)}-G_0^{(2)})\}_{t\ge0}$ and when $\sigma^{(1)}_{(j-1)/l},\sigma^{(2)}_{(j-1)/l},\ j=1,..,[lt]$ are $\mathcal{G}$-measurable, we have
\begin{eqnarray*}
&&\sum_{j=1}^{[lt]}\sigma^{(1)}_{(j-1)/l}\sigma^{(2)}_{(j-1)/l}\frac{1}{\sqrt{n}}\sum_{i\in I_l(j)}\left(p\left(\frac{\Delta_j^nG^{(1)}}{\tau_n^{(1)}}\right)p\left(\frac{\Delta_j^nG^{(2)}}{\tau_n^{(2)}}\right)-\mu_n\right)\\
&\overset{st}{\rightarrow}&\sum_{j=1}^{[lt]}\sigma^{(1)}_{(j-1)/l}\sigma^{(2)}_{(j-1)/l}\sqrt{\beta}(B_{j/l}-B_{(j-1)/l}),\ n\rightarrow\infty\\
&\overset{\mathbb{P}}{\rightarrow}&\sqrt{\beta}\int_0^t\sigma_s^{(1)}\sigma_s^{(2)}dB_s.
\end{eqnarray*}
Hence
$$C_t^{n,l}\overset{\mathcal{L}}{\rightarrow}\sqrt{\beta}\int_0^t\sigma_s^{(1)}\sigma_s^{(2)}dB_s.$$
The convergence of $D_t^n$ follows as in the proof of Proposition 8.5 in \cite{AA2019}.
\end{proof}

\begin{proof}[Proof of Theorem \ref{GBi_CLT}]
We prove the central limit theorem for the generalised class of functions $\phi$ first, and the discussion of choices of the generalisation will be found in Remark \ref{discussgeneralisation}.\\
We only prove for $I(x)=1$ and $I(x)=1_{\{x\ge0\}}$. From Lemma 1 of \cite{BN2011}, we have 
$$\mathbb{E}\left[\phi\left(\frac{\Delta_i^nX^{(j)}}{\tau_n^{(j)}}\right)\right]\le\mathbb{E}\left[\left|\frac{\Delta_i^nX^{(j)}}{\tau_n^{(j)}}\right|^q\right]\le C,$$
$$\mathbb{E}\left[\phi\left(\frac{\sigma_{(i-1)/n}^{(j)}\Delta_i^nG^{(j)}}{\tau_n^{(j)}}\right)\right]\le\mathbb{E}\left[\left|\frac{\sigma_{(i-1)/n}^{(j)}\Delta_i^nG^{(j)}}{\tau_n^{(j)}}\right|^q\right]\le C\ i=1,...,[nt],j=1,2.$$
It is obvious that $\Phi(x,y):=\phi(x)\phi(y)$ is square-integrable w.r.t. any 2-dimensional Gaussian measure, thus all the results established for $V(G,\phi)$ still holds.\\
We will use almost the same method as in the proof of Theorem \ref{Bi_CLT1} to prove the remaining part.\\
 We denote $\mu_n:=\mathbb{E}\left[\Phi\left(\frac{\Delta_1^nG^{(1)}}{\tau_n^{(1)}},\frac{\Delta_1^nG^{(2)}}{\tau_n^{(2)}}\right)\right]$.\\

For a fixed $n$, let $l\le n$, then we have the following decomposition
\begin{eqnarray*}
&&\frac{1}{\sqrt{n}}\sum_{i=1}^{[nt]}\phi\left(\frac{\Delta_i^nX^{(1)}}{\tau_n^{(1)}}\right)\phi\left(\frac{\Delta_i^nX^{(2)}}{\tau_n^{(2)}}\right)-\sqrt{n}\mathbb{E}\left[\Phi\left(\frac{\Delta_1^nG^{(1)}}{\tau_n^{(1)}},\frac{\Delta_1^nG^{(2)}}{\tau_n^{(2)}}\right)\right]\int_0^t\sigma_s^{(1)q}\sigma_s^{(2)q}ds\\
&=&\frac{1}{\sqrt{n}}\sum_{i=1}^{[nt]}\phi\left(\frac{\Delta_i^nX^{(1)}}{\tau_n^{(1)}}\right)\phi\left(\frac{\Delta_i^nX^{(2)}}{\tau_n^{(2)}}\right)-\sqrt{n}\mu_n\int_0^t\sigma_s^{(1)q}\sigma_s^{(2)q}ds
\\
&=&\frac{1}{\sqrt{n}}\sum_{i=1}^{[nt]}\left(\phi\left(\frac{\Delta_i^nX^{(1)}}{\tau_n^{(1)}}\right)\phi\left(\frac{\Delta_i^nX^{(2)}}{\tau_n^{(2)}}\right)-\phi\left(\sigma^{(1)}_{(i-1)/n}\frac{\Delta_i^nG^{(1)}}{\tau_n^{(1)}}\right)\phi\left(\sigma^{(2)}_{(i-1)/n}\frac{\Delta_i^nG^{(2)}}{\tau_n^{(2)}}\right)\right)\\
&&+\frac{1}{\sqrt{n}}\sum_{i=1}^{[nt]}\phi\left(\sigma^{(1)}_{(i-1)/n}\frac{\Delta_i^nG^{(1)}}{\tau_n^{(1)}}\right)\phi\left(\sigma^{(2)}_{(i-1)/n}\frac{\Delta_i^nG^{(2)}}{\tau_n^{(2)}}\right)\\
&&-\frac{1}{\sqrt{n}}\sum_{j=1}^{[lt]}\sigma^{(1)q}_{(j-1)/l}\sigma^{(2)q}_{(j-1)/l}\sum_{i\in I_l(j)}\phi\left(\frac{\Delta_i^nG^{(1)}}{\tau_n^{(1)}}\right)\phi\left(\frac{\Delta_i^nG^{(2)}}{\tau_n^{(2)}}\right)\\
&&+\frac{\sqrt{n}}{l}\mu_n\sum_{j=1}^{[lt]}\sigma_{(j-1)/l}^{(1)q}\sigma_{(j-1)/l}^{(2)q}-\frac{1}{\sqrt{n}}\sum_{i=1}^{[nt]}\sigma_{(i-1)/n}^{(1)q}\sigma_{(i-1)/n}^{(2)q}\\
&&+\frac{1}{\sqrt{n}}\sum_{j=1}^{[lt]}\sigma^{(1)q}_{(j-1)/l}\sigma^{(2)q}_{(j-1)/l}\sum_{i\in I_l(j)}\phi\left(\frac{\Delta_i^nG^{(1)}}{\tau_n^{(1)}}\right)\phi\left(\frac{\Delta_i^nG^{(2)}}{\tau_n^{(2)}}\right)-\frac{\sqrt{n}}{l}\mu_n\sum_{j=1}^{[lt]}\sigma_{(j-1)/l}^{(1)q}\sigma_{(j-1)/l}^{(2)q}\\
&&+\frac{1}{\sqrt{n}}\sum_{i=1}^{[nt]}\sigma_{(i-1)/n}^{(1)q}\sigma_{(i-1)/n}^{(2)q}-\sqrt{n}\mu_n\int_0^t\sigma_s^{(1)q}\sigma_s^{(2)q}ds,\\
\end{eqnarray*}
where $$I_l(j)=\left\{i:\frac{i}{n}\in\left(\frac{j-1}{l},\frac{j}{l}\right]\right\}.$$
Define
\begin{eqnarray*}
A_t^n:&=&\frac{1}{\sqrt{n}}\sum_{i=1}^{[nt]}\left(\phi\left(\frac{\Delta_i^nX^{(1)}}{\tau_n^{(1)}}\right)\phi\left(\frac{\Delta_i^nX^{(2)}}{\tau_n^{(2)}}\right)-\phi\left(\sigma^{(1)}_{(i-1)/n}\frac{\Delta_i^nG^{(1)}}{\tau_n^{(1)}}\right)\phi\left(\sigma^{(2)}_{(i-1)/n}\frac{\Delta_i^nG^{(2)}}{\tau_n^{(2)}}\right)\right),\\
B_t^{'n,l}:&=&\frac{1}{\sqrt{n}}\sum_{i=1}^{[nt]}\phi\left(\sigma^{(1)}_{(i-1)/n}\frac{\Delta_i^nG^{(1)}}{\tau_n^{(1)}}\right)\phi\left(\sigma^{(2)}_{(i-1)/n}\frac{\Delta_i^nG^{(2)}}{\tau_n^{(2)}}\right)\\
&&-\frac{1}{\sqrt{n}}\sum_{j=1}^{[lt]}\sigma^{(1)q}_{(j-1)/l}\sigma^{(2)q}_{(j-1)/l}\sum_{i\in I_l(j)}\phi\left(\frac{\Delta_i^nG^{(1)}}{\tau_n^{(1)}}\right)\phi\left(\frac{\Delta_i^nG^{(2)}}{\tau_n^{(2)}}\right)\\
&=&\frac{1}{\sqrt{n}}\sum_{i=1}^{[nt]}\sigma^{(1)q}_{(i-1)/n}\sigma^{(2)q}_{(i-1)/n}\phi\left(\frac{\Delta_i^nG^{(1)}}{\tau_n^{(1)}}\right)\phi\left(\frac{\Delta_i^nG^{(2)}}{\tau_n^{(2)}}\right)\\
&&-\frac{1}{\sqrt{n}}\sum_{j=1}^{[lt]}\sigma^{(1)q}_{(j-1)/l}\sigma^{(2)q}_{(j-1)/l}\sum_{i\in I_l(j)}\phi\left(\frac{\Delta_i^nG^{(1)}}{\tau_n^{(1)}}\right)\phi\left(\frac{\Delta_i^nG^{(2)}}{\tau_n^{(2)}}\right)\\
B_t^{''n,l}:&=&\frac{\sqrt{n}}{l}\mu_n\sum_{j=1}^{[lt]}\sigma_{(j-1)/l}^{(1)q}\sigma_{(j-1)/l}^{(2)q}-\frac{1}{\sqrt{n}}\sum_{i=1}^{[nt]}\sigma_{(i-1)/n}^{(1)q}\sigma_{(i-1)/n}^{(2)q},\\
C_{t}^{n,l}:&=&\frac{1}{\sqrt{n}}\sum_{j=1}^{[lt]}\sigma^{(1)q}_{(j-1)/l}\sigma^{(2)q}_{(j-1)/l}\sum_{i\in I_l(j)}\phi\left(\frac{\Delta_i^nG^{(1)}}{\tau_n^{(1)}}\right)\phi\left(\frac{\Delta_i^nG^{(2)}}{\tau_n^{(2)}}\right)-\frac{\sqrt{n}}{l}\mu_n\sum_{j=1}^{[lt]}\sigma_{(j-1)/l}^{(1)q}\sigma_{(j-1)/l}^{(2)q},\\
D_t^{n}:&=&\frac{1}{\sqrt{n}}\sum_{i=1}^{[nt]}\sigma_{(i-1)/n}^{(1)q}\sigma_{(i-1)/n}^{(2)q}-\sqrt{n}\mu_n\int_0^t\sigma_s^{(1)q}\sigma_s^{(2)q}ds.\\
\end{eqnarray*}

Notice that since $\sigma^{(j)}$ is bounded on compact interval, $\sigma^{(j)q}$ is still non-negative, c\`adl\`ag and $\eta^{(i)}$-H\"older continuous for $i=1,2$. Thus the proof of convergences of $B_t^{'n,l}+B_t^{''n,l}$, $C_{t}^{n,l}$ and $D_t^{n}$ does not require any changes except that the limit of $C_{t}^{n,l}$ becomes 
$$C_{t}^{n,l}\overset{\mathcal{L}}{\longrightarrow}\sqrt{\beta}\int_0^t\left(\sigma_s^{(1)}\sigma_s^{(2)}\right)^qdB_s.$$

Next, we prove the convergence of $A_t^n$.\\

\begin{eqnarray*}
\mathbb{E}[|A_t^n|]&=&\frac{1}{\sqrt{n}}\mathbb{E}\left[\left|\sum_{i=1}^{[nt]}\left(\phi\left(\frac{\Delta_i^nX^{(1)}}{\tau_n^{(1)}}\right)\phi\left(\frac{\Delta_i^nX^{(2)}}{\tau_n^{(2)}}\right)-\phi\left(\sigma^{(1)}_{(i-1)/n}\frac{\Delta_i^nG^{(1)}}{\tau_n^{(1)}}\right)\phi\left(\sigma^{(2)}_{(i-1)/n}\frac{\Delta_i^nG^{(2)}}{\tau_n^{(2)}}\right)\right)\right|\right]\\
&\le&\frac{1}{\sqrt{n}}\sum_{i=1}^{[nt]}\mathbb{E}\left|\phi\left(\frac{\Delta_i^nX^{(1)}}{\tau_n^{(1)}}\right)\phi\left(\frac{\Delta_i^nX^{(2)}}{\tau_n^{(2)}}\right)-\phi\left(\sigma^{(1)}_{(i-1)/n}\frac{\Delta_i^nG^{(1)}}{\tau_n^{(1)}}\right)\phi\left(\sigma^{(2)}_{(i-1)/n}\frac{\Delta_i^nG^{(2)}}{\tau_n^{(2)}}\right)\right|\\
&=&\frac{1}{2\sqrt{n}}\sum_{i=1}^{[nt]}\mathbb{E}\left|\left(\phi\left(\frac{\Delta_i^nX^{(1)}}{\tau_n^{(1)}}\right)-\phi\left(\sigma^{(1)}_{(i-1)/n}\frac{\Delta_i^nG^{(1)}}{\tau_n^{(1)}}\right)\right)\left(\phi\left(\frac{\Delta_i^nX^{(2)}}{\tau_n^{(2)}}\right)+\phi\left(\sigma^{(2)}_{(i-1)/n}\frac{\Delta_i^nG^{(2)}}{\tau_n^{(2)}}\right)\right)\right.\\
&&+\left.\left(\phi\left(\frac{\Delta_i^nX^{(2)}}{\tau_n^{(2)}}\right)-\phi\left(\sigma^{(2)}_{(i-1)/n}\frac{\Delta_i^nG^{(2)}}{\tau_n^{(2)}}\right)\right)\left(\phi\left(\frac{\Delta_i^nX^{(1)}}{\tau_n^{(1)}}\right)+\phi\left(\sigma^{(1)}_{(i-1)/n}\frac{\Delta_i^nG^{(1)}}{\tau_n^{(1)}}\right)\right)\right|\\
&\le&\frac{1}{2\sqrt{n}}\sum_{i=1}^{[nt]}\mathbb{E}\left|\left(\phi\left(\frac{\Delta_i^nX^{(1)}}{\tau_n^{(1)}}\right)-\phi\left(\sigma^{(1)}_{(i-1)/n}\frac{\Delta_i^nG^{(1)}}{\tau_n^{(1)}}\right)\right)\left(\phi\left(\frac{\Delta_i^nX^{(2)}}{\tau_n^{(2)}}\right)+\phi\left(\sigma^{(2)}_{(i-1)/n}\frac{\Delta_i^nG^{(2)}}{\tau_n^{(2)}}\right)\right)\right|\\
&&+\frac{1}{2\sqrt{n}}\sum_{i=1}^{[nt]}\mathbb{E}\left|\left(\phi\left(\frac{\Delta_i^nX^{(2)}}{\tau_n^{(2)}}\right)-\phi\left(\sigma^{(2)}_{(i-1)/n}\frac{\Delta_i^nG^{(2)}}{\tau_n^{(2)}}\right)\right)\left(\phi\left(\frac{\Delta_i^nX^{(1)}}{\tau_n^{(1)}}\right)
+\phi\left(\sigma^{(1)}_{(i-1)/n}\frac{\Delta_i^nG^{(1)}}{\tau_n^{(1)}}\right)\right)\right|.\\
\end{eqnarray*}
By the inequality $$||x|^q-|y|^q|\le q|x-y|(|x|^{q-1}+|y|^{q-1}),\ \forall q\ge1,$$
and  $$||x1_{\{x\ge0\}}|^q-|y1_{\{y\ge0\}}|^q|\le q|x1_{\{x\ge0\}}-y1_{\{y\ge0\}}|(|x1_{\{x\ge0\}}|^{q-1}+|y1_{\{y\ge0\}}|^{q-1})\le q|x-y|(|x|^{q-1}+|y|^{q-1}),\ \forall q\ge1.$$
We have that $$\left|\phi\left(\frac{\Delta_i^nX^{(1)}}{\tau_n^{(1)}}\right)-\phi\left(\sigma^{(1)}_{(i-1)/n}\frac{\Delta_i^nG^{(1)}}{\tau_n^{(1)}}\right)\right|\le q\left|\frac{\Delta_i^nX^{(1)}}{\tau_n^{(1)}}-\sigma^{(1)}_{(i-1)/n}\frac{\Delta_i^nG^{(1)}}{\tau_n^{(1)}}\right|\left(\left|\frac{\Delta_i^nX^{(1)}}{\tau_n^{(1)}}\right|^{q-1} + \left|\sigma^{(1)}_{(i-1)/n}\frac{\Delta_i^nG^{(1)}}{\tau_n^{(1)}}\right|^{q-1}\right),$$
$$\left|\phi\left(\frac{\Delta_i^nX^{(2)}}{\tau_n^{(2)}}\right)-\phi\left(\sigma^{(2)}_{(i-1)/n}\frac{\Delta_i^nG^{(2)}}{\tau_n^{(2)}}\right)\right|\le q\left|\frac{\Delta_i^nX^{(2)}}{\tau_n^{(2)}}-\sigma^{(2)}_{(i-1)/n}\frac{\Delta_i^nG^{(2)}}{\tau_n^{(2)}}\right|\left(\left|\frac{\Delta_i^nX^{(2)}}{\tau_n^{(2)}}\right|^{q-1} + \left|\sigma^{(2)}_{(i-1)/n}\frac{\Delta_i^nG^{(2)}}{\tau_n^{(2)}}\right|^{q-1}\right).$$
Then by H\"older inequality and Lemma 1 in \cite{BN2011}, we have
\begin{eqnarray*}
\mathbb{E}[|A_t^n|]&\le&\frac{C}{2\sqrt{n}}\sum_{i=1}^{[nt]}\left(\mathbb{E}\left[\left(\frac{\Delta_i^nX^{(1)}}{\tau_n^{(1)}}-\sigma^{(1)}_{(i-1)/n}\frac{\Delta_i^nG^{(1)}}{\tau_n^{(1)}}\right)^2\right]\right)^{\frac{1}{2}}\\
&&+\frac{C}{2\sqrt{n}}\sum_{i=1}^{[nt]}\left(\mathbb{E}\left[\left(\frac{\Delta_i^nX^{(2)}}{\tau_n^{(2)}}-\sigma^{(2)}_{(i-1)/n}\frac{\Delta_i^nG^{(2)}}{\tau_n^{(2)}}\right)^2\right]\right)^{\frac{1}{2}}.\\
\end{eqnarray*}
The remainder of the proof of the convergence is identical to the proof of Theorem \ref{Bi_CLT1}.

\end{proof}

\begin{rmk}\label{discussgeneralisation} We will discuss here which class of functions we can generalise CLT to. First let $\psi: \mathbb{R}\rightarrow \mathbb{R}$ be any non-constant function and define $\Psi(x,y):=\psi(x)\psi(y)$. Then for the central limit theorem for the Gaussian cores, we need that $\mathbb{E}\left[\Psi^2(X)\right]<\infty$ for any $2$-dimensional Gaussian vector $X$.\\
Notice that in the proof of Theorem \ref{Bi_CLT1}, we decompose $$\frac{1}{\sqrt{n}}\sum_{i=1}^{[nt]}p\left(\frac{\Delta_i^nX^{(1)}}{\tau_n^{(1)}}\right)p\left(\frac{\Delta_i^nX^{(2)}}{\tau_n^{(2)}}\right)-\sqrt{n}\mathbb{E}\left[h\left(\frac{\Delta_1^nG^{(1)}}{\tau_n^{(1)}},\frac{\Delta_1^nG^{(2)}}{\tau_n^{(2)}}\right)\right]\int_0^t\sigma_s^{(1)}\sigma_s^{(2)}ds$$ into five different parts, and in the second part
\begin{eqnarray*}
B_t^{'n,l}&=&\frac{1}{\sqrt{n}}\sum_{i=1}^{[nt]}p\left(\sigma^{(1)}_{(i-1)/n}\frac{\Delta_i^nG^{(1)}}{\tau_n^{(1)}}\right)p\left(\sigma^{(2)}_{(i-1)/n}\frac{\Delta_i^nG^{(2)}}{\tau_n^{(2)}}\right)\\
&&-\frac{1}{\sqrt{n}}\sum_{j=1}^{[lt]}\sigma^{(1)}_{(j-1)/l}\sigma^{(2)}_{(j-1)/l}\sum_{i\in I_l(j)}p\left(\frac{\Delta_i^nG^{(1)}}{\tau_n^{(1)}}\right)p\left(\frac{\Delta_i^nG^{(2)}}{\tau_n^{(2)}}\right)\\
&=&\frac{1}{\sqrt{n}}\sum_{i=1}^{[nt]}\sigma^{(1)}_{(i-1)/n}\sigma^{(2)}_{(i-1)/n}p\left(\frac{\Delta_i^nG^{(1)}}{\tau_n^{(1)}}\right)p\left(\frac{\Delta_i^nG^{(2)}}{\tau_n^{(2)}}\right)\\
&&-\frac{1}{\sqrt{n}}\sum_{j=1}^{[lt]}\sigma^{(1)}_{(j-1)/l}\sigma^{(2)}_{(j-1)/l}\sum_{i\in I_l(j)}p\left(\frac{\Delta_i^nG^{(1)}}{\tau_n^{(1)}}\right)p\left(\frac{\Delta_i^nG^{(2)}}{\tau_n^{(2)}}\right),
\end{eqnarray*}
we use the property that $p(ax)=ap(x)$ for any $a\ge0$. For simplicity, we will omit the superscripts in the following discussion. Such factorization property is necessary since we need to separate $\sigma_{(i-1)/n}$ and $p\left(\frac{\Delta_i^nG}{\tau_n}\right)$ to prove the convergence of $B_t^{n,l}$.
Thus, for the function $\psi$, we let it have the following factorization property:
\begin{equation}\label{functionalequation}
\psi(xy)=\alpha(x)\beta(y),\ x\in [0,\infty),y\in\mathbb{R}.
\end{equation}
Here the function $\alpha$ needs to satisfy that $\{\alpha(\sigma_s)\}_{0\le s \le T}$ is $\zeta$-H\"older continuous for $\zeta>1/2$. Considering that $\sigma$ is bounded on the compact interval, we only need to work out what kind of continuity should be satisfied by $\alpha$ on bounded sets. Since uniform continuity is not sufficient, we assume that $\alpha$ is $\xi$-H\"older continuous.
Then, if $\xi>\frac{1}{2\eta}$, where $\sigma$ is $\eta$-H\"older continuous with $\eta>1/2$, $\{\alpha(\sigma_s)\}_{0\le s \le T}$ is $\zeta$-H\"older continuous for $\zeta>1/2$.\\
Now we try to solve functional equation (\ref{functionalequation}). Since $\alpha$ is not the zero function, w.l.o.g., we assume that $\alpha(1)\neq0$, then we fix $x=1$, we have
$$\psi(y)=\alpha(1)\beta(y),$$
which means that $\psi\ \propto\ \beta$ in $\mathbb{R}$.
Let $x,y\in[0,\infty)$, we have
$$\alpha(x)\beta(y)=\psi(xy)=\psi(yx)=\alpha(y)\beta(x).$$
Let $x=1$ in the above equation, we get
$$\beta(y)=\frac{\beta(1)}{\alpha(1)}\alpha(y),$$
which means that $\beta\ \propto\ \alpha$ on $[0,\infty)$. Thus for all $x,y\in[0,\infty)$, there is a constant $K$ such that
$$\psi(xy)=K\psi(x)\psi(y).$$
If we further assume that $\psi$ is normalised and satisfies $\psi(1)=1$, then we have
 \begin{equation}\label{multiplicative}
 \psi(xy)=\psi(x)\psi(y),\ \ x,y\in[0,\infty).
 \end{equation}
The function that satisfies Equation ($\ref{multiplicative}$) is the so-called multiplicative function in number theory. Since we require $\alpha$ to be H\"older continuous, thus $\psi$ is also continuous. By Theorem 3 in Chapter 2.1 of \cite{AC1966}, the complete set of solutions of functional equation (\ref{multiplicative}) is $\{x^p,0\}$. The next step is to extend the domain of $\psi$ from $[0,\infty)$ to $\mathbb{R}$ while the equation (\ref{functionalequation}) still holds. In this step, we can introduce another factor $I(x)=1$ or $1_{\{x\ge0\}}$ or $1_{\{x\le0\}}$ or $sgn(x)$ into function $\psi(x)=|x|^p$. Since we need H\"older continuity of $\alpha$ on a compact set, we will work on the case that $p\ge1$. If we further assume that we only focus on non-negative valued functions, then 
$$\psi(x)=|x|^pI(x),\ I(x)=1, \mathrm{or}\ 1_{\{x\ge0\}}, \mathrm{or}\ 1_{\{x\le0\}},\ p\ge1$$ 
is the only family of functions satisfying Equation (\ref{functionalequation}).\\
Notice that this does not contradict our choice $p(x)=x1_{\{x\ge0\}}$ since $x1_{\{x\ge0\}}=|x|1_{\{x\ge0\}}$. Hence Theorem \ref{GBi_CLT} is indeed a generalisation of Theorem \ref{Bi_CLT1}.\\
\end{rmk}

\clearpage

\addcontentsline{toc}{section}{\bibname} 

\end{document}